\pdfoutput=1
\documentclass[12pt]{amsart}
\usepackage{lmodern}
\usepackage{ifthen}
\usepackage{amsfonts}
\usepackage{amsxtra}
\usepackage{amssymb}
\usepackage{mathdots}
\usepackage{array}
\usepackage[margin=1.0in]{geometry}
\usepackage{xcolor}
\definecolor{cite}{rgb}{0.30,0.60,1.00}
\definecolor{url}{rgb}{0.00,0.00,0.80}
\definecolor{link}{rgb}{0.40,0.10,0.20}
\usepackage[colorlinks,linkcolor=link,urlcolor=url,citecolor=cite,pagebackref,breaklinks]{hyperref}
\usepackage{bbm}
\usepackage{mathtools}
\usepackage{mathrsfs}
\usepackage{appendix}
\usepackage[all]{xy}
\usepackage[lite,abbrev,msc-links,alphabetic]{amsrefs}

\usepackage{graphicx}
\usepackage{multirow}
\usepackage{pstricks}
\usepackage{pst-pdf}
\usepackage{enumitem}
\usepackage{pifont}
\usepackage{marvosym}

\usepackage[OT2,T1]{fontenc}
\DeclareSymbolFont{cyrletters}{OT2}{wncyr}{m}{n}
\DeclareMathSymbol{\Sha}{\mathalpha}{cyrletters}{"58}

\usepackage{graphicx}

\makeatletter
\providecommand*{\Dashv}{%
  \mathrel{%
    \mathpalette\@Dashv\vDash
  }%
}
\newcommand*{\@Dashv}[2]{%
  \reflectbox{$\m@th#1#2$}%
}
\makeatother

\numberwithin{equation}{section}

\theoremstyle{plain}
\newtheorem{proposition}{Proposition}[section]

\newtheorem{lem}[proposition]{Lemma}
\newtheorem{theorem}[proposition]{Theorem}

\theoremstyle{definition}
\newtheorem{definition}[proposition]{Definition}

\newtheorem{notation}[proposition]{Notation}
\newtheorem{assumption}[proposition]{Assumption}

\theoremstyle{remark}
\newtheorem{remark}[proposition]{Remark}

\renewcommand{\b}[1]{\mathbf{#1}}
\renewcommand{\c}[1]{\mathcal{#1}}
\renewcommand{\d}[1]{\mathbb{#1}}
\newcommand{\f}[1]{\mathfrak{#1}}
\renewcommand{\r}[1]{\mathrm{#1}}

\renewcommand{\(}{\left(}
\renewcommand{\)}{\right)}
\newcommand{\res}{\mathbin{|}}

\renewcommand{\leq}{\leqslant}
\renewcommand{\geq}{\geqslant}

\newcommand{\bw}{\b w}

\newcommand{\cB}{\c B}

\newcommand{\cH}{\c H}
\newcommand{\cI}{\c I}

\newcommand{\cL}{\c L}

\newcommand{\cS}{\c S}
\newcommand{\cT}{\c T}

\newcommand{\cV}{\c V}

\newcommand{\cZ}{\c Z}

\newcommand{\dA}{\d A}

\newcommand{\dC}{\d C}

\newcommand{\dQ}{\d Q}

\newcommand{\dZ}{\d Z}

\newcommand{\fS}{\f S}

\newcommand{\fW}{\f W}

\newcommand{\fc}{\f c}

\newcommand{\fe}{\f e}
\newcommand{\ff}{\f f}

\newcommand{\fp}{\f p}

\newcommand{\rI}{\r I}

\newcommand{\rU}{\r U}

\newcommand{\rd}{\,\r d}

\newcommand{\rt}{\r t}

\newcommand{\tc}{\mathtt{c}}

\newcommand{\tp}[1]{\prescript{\rt}{}{#1}}

\newcommand{\CF}{\mathbbm{1}}

\newcommand{\dtm}{\r{det}\:}

\newcommand{\Herm}{\r{Herm}}

\newcommand{\RE}{\r{Re}\,}

\DeclareMathOperator{\BC}{BC}

\DeclareMathOperator{\Gal}{Gal}

\DeclareMathOperator{\GL}{GL}

\DeclareMathOperator{\Ind}{Ind}

\DeclareMathOperator{\Mat}{Mat}

\DeclareMathOperator{\Res}{Res}

\DeclareMathOperator{\Sph}{Sph}

\DeclareMathOperator{\Spec}{Spec}

\DeclareMathOperator{\tr}{tr}
\DeclareMathOperator{\Tr}{Tr}

\DeclareMathOperator{\vol}{vol}

\begin{document}

\title[Theta correspondence for almost unramified representations]
{Theta correspondence for almost unramified representations of unitary groups}

\author{Yifeng Liu}
\address{Institute for Advanced Study in Mathematics, Zhejiang University, Hangzhou 310058, China}
\email{liuyf0719@zju.edu.cn}

\date{\today}

\begin{abstract}
  In this note, we introduce the notion of almost unramified representations of quasi-split unitary groups of even ranks with respect to an unramified quadratic extension of local fields, and study their behavior under the local theta correspondence. These representations are in some sense the least ramified representations that have root number $-1$.
\end{abstract}

\maketitle

\tableofcontents

\section{Introduction}

We fix an unramified quadratic extension $E/F$ of nonarchimedean local fields not of characteristic 2, with $\tc\in\Gal(E/F)$ the Galois involution and $q$ the residue cardinality of $F$. Consider a (nondegenerate) hermitian space $V$ over $E$ (with respect to $\tc$) of even dimension. We may attach a sign $\epsilon=\epsilon(V)\in\{\pm\}$ such that $V$ admits an $O_E$-lattice $\Lambda_V$ satisfying that $\Lambda_V$ is a subgroup of $\Lambda_V^\vee$ of index $q^{1-(\epsilon 1)}$, where $\Lambda_V^\vee$ is the integral dual lattice of $\Lambda_V$.\footnote{Since $E/F$ is unramified, $\Lambda_V$ is necessarily a maximal integral $O_E$-lattice; and moreover, our definition of the sign coincides with the usual one, for example, \cite{HKS96}*{(0.5)}.} Denote by $H_V$ the unitary group of $V$, which is a reductive group over $F$. We fix a lattice $\Lambda_V$ as above, which is unique up to conjugation in $H_V(F)$, and denote its stabilizer in $H_V(F)$ by $L_V$ which is a special maximal compact subgroup. We have the Hecke algebra
\[
\cH_V\coloneqq\dC[L_V\backslash H_V(F)/L_V]
\]
with respect to the Haar measure on $H_V(F)$ that gives $L_V$ volume $1$, which is a commutative complex algebra.

On the other hand, we take an integer $r\geq 0$ and equip $W_r\coloneqq E^{2r}$ with the skew-hermitian form given by the matrix $\(\begin{smallmatrix}&1_r\\ -1_r &\end{smallmatrix}\)$. Denote by $G_r$ the unitary group of $W_r$, which is a reductive group over $F$. Let $K_r\subseteq G_r(F)$ be the stabilizer of the lattice $O_E^{2r}\subseteq W_r$, which is a hyperspecial maximal compact subgroup. For $\epsilon=\pm$, we will define in Section \ref{ss:almost} a certain Hecke algebra $\cH_{W_r}^\epsilon$ of $G_r$ which is commutative, together with an ideal $\cI_{W_r,V}$ of $\cH_{W_r}^{\epsilon(V)}\otimes\cH_V$. When $\epsilon=+$, $\cH_{W_r}^+$ is simply the spherical Hecke algebra $\dC[K_r\backslash G_r(F)/K_r]$ with respect to the Haar measure on $G_r(F)$ that gives $K_r$ volume $1$; when $\dim_EV=2r$ and $\epsilon(V)=+$, $\cH_{W_r}^+$ is canonically isomorphic to $\cH_V$ and $\cI_{W_r,V}$ is simply the diagonal ideal, namely, the kernel of the product homomorphism $\cH_{W_r}^+\otimes\cH_V\to\cH_V$.

Fix a nontrivial additive character $\psi_F\colon F\to\dC^\times$ of conductor $O_F$. Then we have the (complex) Weil representation $(\omega_{W_r,V},\cV_{W_r,V})$ of $G_r(F)\times H_V(F)$ (with respect to the additive character $\psi_F$ and the trivial splitting character). Let $I_r\subseteq K_r$ be a certain Iwahori subgroup. For $\epsilon=\pm$, we will define in Section \ref{ss:almost} a character $\kappa_r^\epsilon$ of the finite Hecke algebra $\dC[I_r\backslash K_r/I_r]$. Denote by $\cS_{W_r,V}$ the subspace of $\cV_{W_r,V}$ consisting of $I_r\times L_V$-invariant vectors on which $\dC[I_r\backslash K_r/I_r]$ acts by the character $\kappa_r^{\epsilon(V)}$. It turns out that $\cS_{W_r,V}$ is a module over $\cH_{W_r}^{\epsilon(V)}\otimes\cH_V$ via the representation $\omega_{W_r,V}$. The following theorem characterizes the structure of this module, which will be proved at the end of this note.

\begin{theorem}\label{th:support}
We have
\begin{enumerate}
  \item The annihilator of the $\cH_{W_r}^{\epsilon(V)}\otimes\cH_V$-module $\cS_{W_r,V}$ is $\cI_{W_r,V}$.

  \item When $q$ is odd, $\cS_{W_r,V}$ is free over $(\cH_{W_r}^{\epsilon(V)}\otimes\cH_V)/\cI_{W_r,V}$, and is generated by the characteristic function of $\Lambda_V^r$ under the Schr\"{o}dinger model $\cV_{W_r,V}\simeq C^\infty_c(V^r)$.
\end{enumerate}
\end{theorem}

Now we introduce the notion of almost unramified representations. Recall that an irreducible admissible representation $\pi$ of $G_r(F)$ satisfying $\pi^{I_r}\neq\{0\}$ is a constituent of an unramified principal series. Thus, we may associate an element in $(\dC^\times)^r$, unique up to the action of the Weyl group $\{\pm1\}^r\rtimes\fS_r$, to $\pi$, which we call the \emph{Satake parameter} of $\pi$. By our definition of $\kappa_r^+$, which is simply the character on the characteristic function of $K_r$, $\pi$ is unramified (with respect to $K_r$) if and only if $\pi^{I_r}[\kappa_r^+]\neq\{0\}$. Now we say that an irreducible admissible representation $\pi$ of $G_r(F)$ is \emph{almost unramified} (with respect to $K_r$) if $\pi^{I_r}[\kappa_r^-]\neq\{0\}$ and that the Satake parameter of $\pi$ contains either $q$ or $q^{-1}$ (Definition \ref{de:representation}).

For every element $\sigma\in(\dC/\tfrac{\pi i}{\log q}\dZ)^r$, let $\rI^\sigma_{W_r}$ be the corresponding unramified principal series of $G_r(F)$, and $\pi^\sigma_{W_r,\epsilon}$ its unique (irreducible) constituent satisfying $(\pi^\sigma_{W_r,\epsilon})^{I_r}[\kappa_r^\epsilon]\neq\{0\}$. Then every almost unramified representation of $G_r(F)$ is isomorphic to $\pi^\sigma_{W_r,-}$ for some $\sigma$ that contains either $\tfrac{1}{2}$ or $-\tfrac{1}{2}$. See Section \ref{ss:proof} for more details on principal series.

On the other hand, let $V_\epsilon$ be the unique up to isomorphism hermitian space over $E$ of dimension $2r$ and sign $\epsilon$. For every element $\sigma\in(\dC/\tfrac{\pi i}{\log q}\dZ)^{r-1}$, we have the corresponding principal series $\rI^\sigma_{V_-}$ of $H_{V_-}(F)$, with $\pi_{V_-}^\sigma$ its unique (irreducible) constituent satisfying $(\pi^\sigma_{V_-})^{L_{V_-}}\neq\{0\}$. We have the following theorem concerning the theta lifting of almost unramified representations.

\begin{theorem}[special case of Theorem \ref{th:theta}]\label{th:theta_pre}
Suppose that $q$ is odd. For every $\sigma\in(\dC/\tfrac{\pi i}{\log q}\dZ)^{r-1}$, we have
\[
\theta(\pi^{(\sigma,1/2)}_{W_r,-},V_-)\simeq\pi^\sigma_{V_-},
\]
where the left-hand side denotes the theta lifting of $\pi^{(\sigma,1/2)}_{W_r,-}$ to the hermitian space $V_-$ (with respect to the additive character $\psi_F$ and the trivial splitting character). In particular, we have $\theta(\pi^{(\sigma,1/2)}_{W_r,-},V_+)=0$ by the theta dichotomy \cite{GG11}.
\end{theorem}

In the course of proving Theorem \ref{th:support} and Theorem \ref{th:theta_pre}, we have to compute the doubling $L$-factor $L(s,\pi)$ \cites{HKS96,Yam14} for an almost unramified representation $\pi$ of $G_r(F)$, which we state in the following result.

\begin{theorem}[Theorem \ref{th:gcd}]\label{th:gcd_pre}
Let $\pi$ be an almost unramified representation of $G_r(F)$ that is isomorphic to $\pi^{(\sigma,1/2)}_{W_r,-}$ for some $\sigma=(\sigma_1,\dots,\sigma_{r-1})\in(\dC/\tfrac{\pi i}{\log q}\dZ)^{r-1}$. Then we have
\[
L(s,\pi)=\frac{1}{1-q^{-1-2s}}\cdot\prod_{i=1}^{r-1}\frac{1}{(1-q^{2\sigma_i-2s})(1-q^{-2\sigma_i-2s})},
\]
and $\varepsilon(s,\pi,\psi_F)=-q^{1-2s}$ for the doubling epsilon factor. In particular, $\pi$ is not unramified.
\end{theorem}

We also explicitly compute the doubling zeta integral for a test vector in $\pi^{I_r}[\kappa_r^-]$ in Proposition \ref{pr:zeta}, as a necessary step toward the proof of Theorem \ref{th:gcd}. The formula itself will have applications in obtaining some explicit (arithmetic) Rallis inner product formulae.

\begin{remark}
We have the following remarks.
\begin{enumerate}
  \item Suppose that $\pi$ is an unramified irreducible admissible representation of $G_r(F)$. Then the analogous result in Theorem \ref{th:theta_pre} has been established in \cite{Liu11}*{Appendix}, closely following the argument of Rallis \cite{Ral82}.

  \item By Theorem \ref{th:theta_pre}, when $q$ is odd, \emph{almost} unramified representations are exactly those representations whose theta lifting to the non-quasi-split unitary group of the same (even) rank has nonzero invariants under the stabilizer of an \emph{almost} self-dual lattice, which justifies the terminology. In fact, these representations already appeared in the work \cite{LTXZZ}.

  \item Suppose that $\pi$ is an almost unramified irreducible admissible representation of $G_r(F)$. By Theorem \ref{th:gcd_pre} and \cite{LTXZZ}*{Lemma~C.2.3(2)}, we have the comparison
       \[
       L(s,\pi)=L(s,\BC(\pi)),\quad\varepsilon(s,\pi,\psi_F)=\varepsilon(s,\BC(\pi),\psi_F)
       \]
       with the standard base change $L$-factor and epsilon factor, which confirms the expectation in \cite{Yam14}*{Remark~5.1}. Moreover, the correspondence in Theorem \ref{th:theta_pre} verifies the epsilon dichotomy expected in \cite{HKS96} (see \cite{HKS96}*{Theorem~6.1} for the supercuspidal case) as $\varepsilon(\tfrac{1}{2},\pi,\psi_F)=-1$.

  \item By \cite{LTXZZ}*{Lemma~C.2.3(2) \& Lemma~C.2.5}, we know that the correspondence in Theorem \ref{th:theta_pre} or more generally in Theorem \ref{th:theta} is functorial, as new (affirmative) examples toward the Adams' conjecture \cite{HKS96}*{Conjecture~7.2}.

  \item Both the epsilon dichotomy mentioned in (1) and the functoriality of the theta correspondence in Theorem \ref{th:theta} when $\dim_EV=2r$ have already been established in \cite{GS12}.\footnote{In \cite{GS12}, the authors studied the symplectic-orthogonal case; but their argument works for the unitary case as well.}

  \item The sole reason for us to assume $q$ odd in Theorem \ref{th:support}(2) and Theorem \ref{th:theta_pre} is that we can only prove Assumption \ref{as:invariant} when $q$ is odd, since the proof uses Waldspurger's generalized lattice model. We expect that those statements remain true when $q$ is even.
\end{enumerate}
\end{remark}

Our motivation for studying theta correspondence of almost unramified representations comes from global aspect, namely the arithmetic theta lifting or the arithmetic Rallis inner product formula \cite{Liu11}; see our very recent work \cite{LL}. In the global setting, $E/F$ will be a CM extension of number fields; and we consider a cuspidal automorphic representation $\pi$ of $G_r(\dA_F)$ such that its archimedean component is the holomorphic discrete series of the ``minimal'' weight. Under the current technique, it is necessary to assume that $E/F$ is everywhere unramified, which forces $[F:\dQ]$ even. In particular, the product of archimedean root numbers $\varepsilon(\pi_\infty)$ is always $1$. Thus, if we want to put ourselves in the arithmetic situation, that is, the global root number $\varepsilon(\pi)$ is $-1$, then it is necessary to have at least one finite place $v$ of $F$ inert in $E$ such that $\varepsilon(\pi_v)=-1$. The almost unramified representations studied in this note are in some sense least ramified among all representations with root number $-1$.

In the course of writing, although results in the case where $\epsilon=+$ are essentially known from previous literatures, we will treat the two cases in a parallel way so that readers can see the analogy very clearly.

\subsubsection*{Notations and conventions}

\begin{itemize}
  \item When we have a function $f$ on a product set $A_1\times\cdots\times A_n$, we will write $f(a_1,\dots,a_n)$ instead of $f((a_1,\dots,a_n))$ for its value at an element $(a_1,\dots,a_n)\in A_1\times\cdots\times A_n$.

  \item For a set $S$, we denote by $\CF_S$ the characteristic function of $S$.

  \item Tensor product of complex vector spaces over the field of complex numbers will simply be denoted as $\otimes$.

  \item For an integer $r\geq 0$, we denote by $\fS_r$ the group of permutations of $\{1,\dots,r\}$ and put $\fW_r\coloneqq\{\pm1\}^r\rtimes\fS_r$. For $1\leq i\leq r$, we denote by $w_i$ the element $-1$ in the $i$-th factor of $\{\pm 1\}^r$; for every element $\sigma\in\fS_r$, we put $w'_\sigma\coloneqq(1^r,\sigma)\in\fW_r$. Note that $\fW_r$ is generated by $\{w_1,w'_{(1,2)},\dots,w'_{(r-1,r)}\}$.

  \item For an integer $r\geq 0$, we put $\cT_r\coloneqq\dC[T_1^{\pm 1},\dots,T_r^{\pm 1}]^{\fW_r}$, where $\fW_r$ acts in the way that $w_i$ turns $T_i^\pm$ to $T_i^\mp$ and $w'_\sigma$ permutes the variables.

  \item Recall that we have fixed an unramified quadratic extension $E/F$ of nonarchimedean local fields not of characteristic 2, with $\tc\in\Gal(E/F)$ the Galois involution. Let $\fp_F$ and $\fp_E$ be the maximal ideal of $O_F$ and $O_E$, respectively; and put $q\coloneqq|O_F/\fp_F|$.

  \item For an integer $r\geq 0$, we denote by $0_r$ and $1_r$ the null and identity matrices of rank $r$, respectively, and $\Herm_r$ the subscheme of $\Res_{E/F}\Mat_{r,r}$ of $r$-by-$r$ matrices $b$ satisfying $\tp{b}^\tc=b$.

  \item We fix a (nontrivial) additive character $\psi_F\colon F\to\dC^\times$ of conductor $\fp_F^\fc$ for some $\fc\in\dZ$. Put $\psi_E\coloneqq\psi_F\circ\Tr_{E/F}\colon E\to\dC^\times$.

  \item Denote by $|\;|_E\colon E^\times\to\dC^\times$ the normalized norm character which sends a uniformizer to $q^{-2}$.
\end{itemize}

\subsubsection*{Acknowledgements}

The author would like to thank the anonymous referee for careful reading and helpful comments. The research of the author is partially supported by the NSF grant DMS--1702019.

\section{Unramified and almost unramified Hecke algebras}
\label{ss:almost}

In this section, we will define several Hecke algebras and some objects associated to them.

Let $r\geq 0$ be an integer. We equip $W_r\coloneqq E^{2r}$ with the skew-hermitian form given by the matrix
\[
\begin{pmatrix}
 & 1_r \\
-1_r &  \\
\end{pmatrix}
.
\]
We denote by $\{e_1,\dots,e_{2r}\}$ the natural basis of $W_r$. Denote by $G_r$ the unitary group of $W_r$, which is a reductive group over $F$. We write elements of $W_r$ in the row form, on which $G_r$ acts from the right.

Let $K_r\subseteq G_r(F)$ be the stabilizer of the lattice $O_E^{2r}\subseteq W_r$, which is a hyperspecial maximal compact subgroup. Let $P_r$ be the Borel subgroup of $G_r$ consisting of elements of the form
\[
\begin{pmatrix}
a & b \\
& \tp{a}^{\tc,-1} \\
\end{pmatrix}
,
\]
in which $a$ is a \emph{lower-triangular} matrix in $\Res_{E/F}\GL_r$.\footnote{Our choice of the Borel subgroup is consistent with \cite{Ral82}*{Section~4} and \cite{Liu11}*{Appendix}, but not consistent with \cite{GPR}*{Part~A} and \cite{Li92}.} Let $M_r$ be the standard diagonal Levi factor of $P_r$, and $I_r\subseteq K_r$ the Iwahori subgroup corresponding to $P_r$. For every integer $0\leq t\leq r$, let $P_r^t$ be the maximal parabolic subgroup of $G_r$ containing $P_r$ with the unipotent radical $N_r^t$, such that the standard diagonal Levi factor $M_r^t$ of $P_r^t$ is isomorphic to $G_t\times\Res_{E/F}\GL_{r-t}$; and also let $I_r^t\subseteq K_r$ be the corresponding parahoric subgroup. In particular, $P_r^r=G_r$ and $P_r^0$ is the (hermitian Siegel) parabolic subgroup stabilizing $Y_r$. In the discussion later, we will use the isomorphism
\begin{align}\label{eq:levi}
m\colon G_t\times\Res_{E/F}\GL_{r-t}\xrightarrow{\sim}M_r^t
\end{align}
given by the assignment
\[
\(
\begin{pmatrix}
    g_{11} & g_{12} \\
    g_{21} & g_{22} \\
\end{pmatrix}
,a
\)
\mapsto
\begin{pmatrix}
   g_{11} && g_{12} & \\
   &  a &&  \\
    g_{21} && g_{22} & \\
   &&  & \tp{a}^{\tc,-1} \\
\end{pmatrix}
.
\]
We also have an isomorphism
\[
n\colon\Herm_r\xrightarrow{\sim} N_r^0
\]
given by the assignment
\[
b \mapsto
\begin{pmatrix}
    1_r & b \\
     & 1_r \\
\end{pmatrix}
.
\]

We now explicitly realize the Weyl group of $G_r$, which is isomorphic to $\fW_r=\{\pm1\}^r\rtimes\fS_r$, as a subgroup of $K_r$ in the way that
\[
w_1=
\begin{pmatrix}
    && 1 & \\
   &  1_{r-1} &&  \\
   -1 &&  & \\
   &&  & 1_{r-1} \\
\end{pmatrix}
,
\]
and for $\sigma\in\fS_r$, $w'_\sigma$ corresponds to the element in $K_r$ that permutes $\{e_1,\dots,e_r\}$ by $\sigma$. We have the Bruhat decomposition
\[
K_r=\coprod_{w\in\fW_r}I_rwI_r,
\]
and the finite Hecke algebra $\dC[I_r\backslash K_r/I_r]$ with the identity element $\CF_{I_r}$.

\begin{definition}\label{de:kappa}
For $\epsilon=\pm$, we define a character $\kappa_r^\epsilon$ of $\dC[I_r\backslash K_r/I_r]$ uniquely determined by the condition that $\kappa_r^\epsilon$ sends $\CF_{I_rw_1I_r}$ to $q(-q)^{((\epsilon1)-1)/2}$ and $\CF_{I_rw'_{(i,i+1)}I_r}$ to $q^2$ for $1\leq i\leq r-1$, which is well-defined by \cite{Mat64}.
\end{definition}

We regard $\dC[I_r\backslash K_r/I_r]$ as a module over itself via the right convolution.

\begin{lem}\label{le:eigenvector}
The eigenspace $\dC[I_r\backslash K_r/I_r][\kappa_r^\epsilon]$ is spanned over $\dC$ by the function
\[
\fe_r^\epsilon\coloneqq\sum_{i=0}^r (-q)^{\frac{((\epsilon1)-1)i}{2}}\CF_{I_r^0w_1\cdots w_i I_r^0}.
\]
\end{lem}

\begin{proof}
In both cases, the condition that $\kappa_r^\epsilon$ sends $\CF_{I_rw'_{(i,i+1)}I_r}$ to $q^2$ for $1\leq i\leq r-1$ implies that $\dC[I_r\backslash K_r/I_r][\kappa_r^\epsilon]\subseteq\dC[I_r^0\backslash K_r/I_r^0]$.

By the Bruhat decomposition $K_r=\coprod_{i=0}^rI_r^0w_1\cdots w_i I_r^0$, every element in $\dC[I_r\backslash K_r/I_r][\kappa_r^\epsilon]$ has the form $\sum_{i=0}^r c_i\CF_{I_r^0w_1\cdots w_i I_r^0}$. Now for $0\leq j<r$, the coefficient of $\CF_{I_rw_1\cdots w_{j+1} I_r}$ in $\CF_{I_rw_1I_r}\cdot\sum_{i=0}^r c_i\CF_{I_r^0w_1\cdots w_i I_r^0}$ is same as in $\CF_{I_rw_1 I_r}\cdot\(c_j\CF_{I_r w_2\cdots w_{j+1}I_r}+c_{j+1}\CF_{I_rw_1\cdots w_{j+1} I_r}\)$, which is $c_j+(q-1)c_{j+1}$. Thus, when $\epsilon=+$, we must have $c_j+(q-1)c_{j+1}=qc_{j+1}$, namely, $c_j=c_{j+1}$ for $0\leq j <r$; when $\epsilon=-$, we must have $c_j+(q-1)c_{j+1}=-c_{j+1}$, namely, $c_j=(-q)c_{j+1}$ for $0\leq j <r$. The lemma follows.
\end{proof}

\begin{remark}\label{re:eigenvector}
We may write the function $\fe_r^\epsilon$ in Lemma \ref{le:eigenvector} in a slightly more intrinsic way. We have the Bruhat decomposition $K_r=\coprod_{i=0}^r\cB_i$ into bi-$I_r^0$-invariant subsets $\cB_i$ of $K_r$ such that $\cB_i\prec\cB_j$ in the Bruhat order for $0\leq i<j\leq r$. Then $\fe_r^\epsilon=\sum_{i=0}^r(-q)^{\frac{((\epsilon1)-1)i}{2}}\CF_{\cB_i}$.
\end{remark}

\begin{remark}
The decomposition of the $\dC[I_r\backslash K_r/I_r]$-module $\dC[I_r\backslash K_r/I_r]$ is controlled by representations of $\fW_r$ (see, for example, \cite{CIK71}*{Corollary~2.2} with $J=\emptyset$). In fact, from the proof of \cite{CIK71}*{Theorem~2.1}, one sees that the eigenspace $\dC[I_r\backslash K_r/I_r][\kappa_r^+]$ corresponds to the trivial character of $\fW_r$; and the eigenspace $\dC[I_r\backslash K_r/I_r][\kappa_r^-]$ corresponds to the character that is the (unique) extension of the product character $\{\pm 1\}^r\to\{\pm 1\}$, which is invariant under the $\fS_r$-action, to $\fW_r$ that is trivial on $\{+1\}^r\rtimes\fS_r$.
\end{remark}

Let $e_r^\epsilon$ be the unique nonzero element in the eigenspace $\dC[I_r\backslash K_r/I_r][\kappa_r^\epsilon]$, which is one dimensional by Lemma \ref{le:eigenvector}, satisfying $(e_r^\epsilon)^2=e_r^\epsilon$. We denote by $\cH^{I_r}_{W_r}\coloneqq\dC[I_r\backslash G_r(F)/I_r]$ the Iwahori-Hecke algebra, with respect to the Haar measure on $G_r(F)$ that gives $I_r$ volume $1$, which contains its center $Z(\cH^{I_r}_{W_r})$ and the finite subalgebra $\dC[I_r\backslash K_r/I_r]$.

\begin{definition}
For $\epsilon=\pm$, we define $\cH^\epsilon_{W_r}$ to be the algebra $e_r^\epsilon Z(\cH^{I_r}_{W_r})$ contained in $\cH^{I_r}_{W_r}$, with the algebra structure inherited from $\cH^{I_r}_{W_r}$. We call $\cH^+_{W_r}$ and $\cH^-_{W_r}$ the \emph{unramified} and \emph{almost unramified Hecke algebras} of $G_r$, respectively.
\end{definition}

\begin{lem}
For $\epsilon=\pm$, the assignment $f\mapsto e_r^\epsilon f$ induces an isomorphism
\[
\beta_r^\epsilon\colon Z(\cH^{I_r}_{W_r})\xrightarrow{\sim}\cH^\epsilon_{W_r}
\]
of algebras.
\end{lem}

\begin{proof}
It suffices to show that the assignment $f\mapsto e_r^\epsilon f$ is injective. By Bernstein's presentation of $\cH^{I_r}_{W_r}$, the natural map $\dC[I_r\backslash K_r/I_r]\otimes  Z(\cH^{I_r}_{W_r})\to\cH^{I_r}_{W_r}$ is injective \cite{HKP}*{Lemma~1.7.1 and Lemma~2.3.1}. The lemma follows.
\end{proof}

\begin{lem}
The assignment $f\mapsto [K_r:I_r]f$ induces an isomorphism
\[
\cH^+_{W_r}\xrightarrow{\sim}\dC[K_r\backslash G_r(F)/K_r],
\]
in which the convolution on the target is with respect to the Haar measure on $G_r(F)$ that gives $K_r$ volume $1$.
\end{lem}

In what follows, we identify $\cH^+_{W_r}$ with $\dC[K_r\backslash G_r(F)/K_r]$.

\begin{proof}
By Lemma \ref{le:eigenvector}, $e_r^+$ is nothing but $\tfrac{1}{[K_r:I_r]}\CF_{K_r}$. Thus, the lemma follows from the Bernstein isomorphism (see \cite{HKP}*{Section~4.6}).
\end{proof}

Now let $(V,(\;,\;)_V)$ be a hermitian space over $E$ of \emph{even} dimension. Let $d=d(V)\coloneqq\tfrac{1}{2}\dim_EV$ be the half dimension, $s=s(V)\geq 0$ the Witt index, and $\epsilon=\epsilon(V)\in\{\pm\}$ the sign of $V$, respectively, satisfying $2d-2s=1-(\epsilon1)$. We also let $H_V$ be the unitary group of $V$. We fix an $O_E$-lattice $\Lambda_V$ of $V$ such that $\Lambda_V$ is a subgroup of $\Lambda_V^\vee$ of index $q^{1-(\epsilon 1)}$, where
\[
\Lambda_V^\vee\coloneqq\{x\in V\res \psi_E((x,y)_V)=1\text{ for every }y\in\Lambda_V\}
\]
is the $\psi_E$-dual lattice of $\Lambda_V$. Note that such lattice is unique up to conjugation in $H_V(F)$. Let $L_V\subseteq H_V(F)$ be the stabilizer of $\Lambda_V$, which is a special maximal compact subgroup. We have the Hecke algebra
\[
\cH_V\coloneqq\dC[L_V\backslash H_V(F)/L_V]
\]
with respect to the Haar measure on $H_V(F)$ that gives $L_V$ volume $1$. Since $L_V$ is special, $\cH_V$ is a commutative complex algebra by the Satake isomorphism (see, for example, \cite{Car79}*{Corollary~4.1}).

\begin{definition}\label{de:phi}
We define
\begin{itemize}
  \item the homomorphism
    \[
    \overleftarrow\Theta\colon\cH^\epsilon_{W_r}\to\cT_{\min\{r,s\}}
    \]
    to be the composition of the isomorphism $\beta_r^+\circ(\beta_r^\epsilon)^{-1}\colon\cH^\epsilon_{W_r}\xrightarrow{\sim}\dC[K_r\backslash G_r(F)/K_r]$, the canonical isomorphism $\dC[K_r\backslash G_r(F)/K_r]\simeq\cT_r$, and the homomorphism $\cT_r\to\cT_{\min\{r,s\}}$ sending $F(T_1^{\pm 1},\dots,T_r^{\pm 1})$ to
    \[
    F(T_1^{\mp 1},\dots,T_{\min\{r,s\}}^{\mp 1},q^{\pm(\epsilon 1)},q^{\pm(2+(\epsilon1))},\dots,q^{\pm(2(r-\min\{r,s\}-1)+(\epsilon1))});
    \]

  \item the homomorphism
    \[
    \overrightarrow\Theta\colon\cH_V\to\cT_{\min\{r,s\}}
    \]
    to be the composition of the canonical isomorphism $\cH_V\simeq\cT_s$, and the homomorphism $\cT_s\to\cT_{\min\{r,s\}}$ sending $F(T_1^{\pm 1},\dots,T_s^{\pm 1})$ to
    \[
    F(q^{\pm(2-(\epsilon1))},q^{\pm(4-(\epsilon1))},\dots,q^{\pm(2(s-\min\{r,s\})-(\epsilon1))},T_1^{\pm 1},\dots,T_{\min\{r,s\}}^{\pm 1}).
    \]
\end{itemize}
We then define $\cI_{W_r,V}$ to be the kernel of the induced homomorphism
\[
\overleftarrow\Theta\otimes\overrightarrow\Theta\colon\cH^\epsilon_{W_r}\otimes\cH_V\to\cT_{\min\{r,s\}}.
\]
\end{definition}

\begin{remark}\label{re:phi}
Note that at least one of $\overleftarrow\Theta$ and $\overrightarrow\Theta$ is an isomorphism, which implies that $(\cH^\epsilon_{W_r}\otimes\cH_V)/\cI_{W_r,V}$ is isomorphic to $\cT_{\min\{r,s\}}$. In particular, $(\cH^\epsilon_{W_r}\otimes\cH_V)/\cI_{W_r,V}$ is a smooth commutative complex algebra.
\end{remark}

\section{Weil representations and the spherical module}
\label{ss:weil}

In this section, we review the Schr\"{o}dinger model of the Weil representation, introduce the spherical module, and prove several properties. Let $(V,(\;,\;)_V)$ be a hermitian space over $E$ of even dimension, with $d=d(V)$ and $\epsilon=\epsilon(V)$. Recall that we have fixed an $O_E$-lattice $\Lambda_V$ of $V$ satisfying that $\Lambda_V$ is a subgroup of $\Lambda_V^\vee$ of index $q^{1-(\epsilon 1)}$.

For an element $x=(x_1,\dots,x_r)\in V^r$, we denote by
\[
T(x)\coloneqq\((x_i,x_j)_V\)_{1\leq i,j\leq r}\in\Herm_r(F)
\]
the moment matrix of $x$. Put $\Sigma_r(V)\coloneqq\{x\in V^r\res T(x)=0_r\}$. We have the Fourier transform $C^\infty_c(V^r)\to C^\infty_c(V^r)$ sending $\phi$ to $\widehat\phi$ defined by the formula
\[
\widehat\phi(x)\coloneqq\int_{V^r}\phi(y)\psi_E\(\sum_{i=1}^r(x_i,y_i)_V\)\rd y,
\]
where $\r{d}y$ is the self-dual Haar measure on $V^r$ with respect to $\psi_E$.

Let $(\omega_{W_r,V},\cV_{W_r,V})$ be the Weil representation of $G_r(F)\times H_V(F)$ (with respect to the additive character $\psi_F$ and the trivial splitting character). We recall the action under the Schr\"{o}dinger model $\cV_{W_r,V}\simeq C^\infty_c(V^r)$ as follows:
\begin{itemize}
  \item for $a\in\GL_r(E)$ and $\phi\in C^\infty_c(V^r)$, we have
     \[
     \omega_{W_r,V}(m(a))\phi(x)=|\dtm a|_E^d\cdot \phi(x a);
     \]

  \item for $b\in\Herm_r(F)$ and $\phi\in C^\infty_c(V^r)$, we have
     \[
     \omega_{W_r,V}(m(b))\phi(x)=\psi_F(\tr bT(x))\cdot \phi(x);
     \]

  \item for $\phi\in C^\infty_c(V^r)$, we have
     \[
     \omega_{W_r,V}\(\(\begin{smallmatrix} & 1_r \\ -1_r & \end{smallmatrix}\)\)\phi(x)=(\epsilon1)^r\cdot\widehat\phi(x);
     \]

  \item for $h\in H_V(F)$ and $\phi\in C^\infty_c(V^r)$, we have
     \[
     \omega_{W_r,V}(h)\phi(x)=\phi(h^{-1}x).
     \]
\end{itemize}

\begin{definition}
We define the \emph{spherical module}\footnote{Warning: when $\epsilon=-$, elements in the spherical module are in general \emph{not} spherical with respect to any special maximal compact subgroup of $G_r(F)$.} $\cS_{W_r,V}$ to be the subspace of $\cV_{W_r,V}$ consisting of $I_r\times L_V$-invariant vectors on which $\dC[I_r\backslash K_r/I_r]$ acts by the character $\kappa_r^\epsilon$ (Definition \ref{de:kappa}), as a module over $\cH_{W_r}^\epsilon\otimes\cH_V$ via the representation $\omega_{W_r,V}$. We denote by $\Sph(V^r)$ the corresponding subspace of $C^\infty_c(V^r)$ under the Schr\"{o}dinger model.
\end{definition}

\begin{lem}\label{le:generator}
The function $\CF_{\Lambda_V^r}$ belongs to $\Sph(V^r)$.
\end{lem}

\begin{proof}
The case where $\epsilon=+$ is easy and well-known.

Now we consider the case where $\epsilon=-$. Since $L_V$ is the stabilizer of $\Lambda_V$, $\CF_{\Lambda_V^r}$ is fixed by $L_V$. As $I_r^0$ acts trivially on $\CF_{\Lambda_V^r}$, we have
\[
\omega_{W_r,V}(\CF_{I_rw'_{(i,i+1)}I_r})(\CF_{\Lambda_V^r})=[I_rw'_{(i,i+1)}I_r:I_r]\CF_{\Lambda_V^r}=q^2\CF_{\Lambda_V^r}
\]
for every $1\leq i\leq r-1$ as $w'_{(i,i+1)}\in I_r^0$. It remains to show that
\begin{align}\label{eq:generator}
\omega_{W_r,V}(\CF_{I_rw_1I_r})(\CF_{\Lambda_V^r})=-\CF_{\Lambda_V^r}.
\end{align}
However, we have
\[
\omega_{W_r,V}(\CF_{I_rw_1I_r})(\CF_{\Lambda_V^r})=\omega_{W_r,V}(\CF_{I_rw_1I_r})(\CF_{\Lambda_V}\otimes\CF_{\Lambda_V^{r-1}})
=\omega_{W_1,V}(\CF_{I_1w_1I_1})(\CF_{\Lambda_V})\otimes\CF_{\Lambda_V^{r-1}},
\]
and
\[
\omega_{W_1,V}(\CF_{I_1w_1I_1})(\CF_{\Lambda_V})(x)=-\sum_{b\in O_F/\fp_F}\psi_F(b(x,x)_V)\widehat{\CF_{\Lambda_V}}(x).
\]
As $\widehat{\CF_{\Lambda_V}}=q^{-1}\CF_{\Lambda_V^\vee}$, we have
\[
\omega_{W_1,V}(\CF_{I_1w_1I_1})(\CF_{\Lambda_V})(x)=-\CF_{\Lambda_V}
\]
since $\Lambda_V=\{x\in\Lambda_V^\vee\res\psi_F((x,x)_V)=1\}$. Thus, \eqref{eq:generator} holds and the lemma follows.
\end{proof}

\begin{assumption}\label{as:invariant}
The representation $C^\infty_c(V^r)^{I_r}[\kappa_r^\epsilon]$ of $H_V(F)$ is generated by $\CF_{\Lambda_V^r}$.
\end{assumption}

\begin{proposition}\label{pr:invariant}
When $q$ is odd, Assumption \ref{as:invariant} holds.
\end{proposition}

\begin{proof}
When $\epsilon=+$, this is proved in \cite{How79}*{Theorem~10.2}. Now we consider the case where $\epsilon=-$. The proof uses Waldspurger's generalized lattice model in \cite{Wal90}, which is the reason we assume $q$ odd.

Let $\widetilde{F}$ and $\widetilde{E}$ be the residue fields of $F$ and $E$, respectively. Let $\widetilde{W}_r$ be the analogue of $W_r$ over $\widetilde{E}$. Fix a uniformizer $\varpi$ of $F$. Then we may equip $\widetilde{V}\coloneqq\Lambda_V^\vee/\Lambda_V$ with an $\widetilde{E}$-valued hermitian form $(\;,\;)_{\widetilde{V}}$ such that $\langle\widetilde{x},\widetilde{y}\rangle_{\widetilde{V}}$ equals the reduction of $\varpi^{1-\fc}\langle x,y\rangle_V$ in $\widetilde{E}$ where $x$ and $y$ are arbitrary liftings of $\widetilde{x}$ and $\widetilde{y}$ to $\Lambda_V^\vee$, respectively. Then $(\widetilde{V},(\;,\;)_{\widetilde{V}})$ is a hermitian space over $\widetilde{E}$ of dimension $1$. Let $(\omega_{\widetilde{W}_r,\widetilde{V}},\cV_{\widetilde{W}_r,\widetilde{V}})$ be the Weil representation of $\rU(\widetilde{W}_r)\times\rU(\widetilde{V})$, hence of $K_r$ as well by inflation. We claim that
\begin{itemize}
  \item [(a)] $\cV_{\widetilde{W}_r,\widetilde{V}}^{I_r}$ is of dimension one, on which $\dC[I_r\backslash K_r/I_r]$ acts via the character $\kappa_r^-$.
\end{itemize}

Assuming the claim, let $\rho$ be the unique irreducible representation of $K_r$ contained in $\cV_{\widetilde{W}_r,\widetilde{V}}$ such that $\rho^{I_r}\neq\{0\}$. Using the lattice $\Lambda_V\otimes_{O_E}O_E^{2r}$ of $V\otimes_E W_r$, we may realize the Weil representation $\omega_{W_r,V}$ in the space $\cL(\Lambda_V)$ of locally constant, compactly supported functions from $V\otimes_E W_r$ to $\cV_{\widetilde{W}_r,\widetilde{V}}$ satisfying a certain transformation law under the translation by $\Lambda_V^\vee\otimes_{O_E}O_E^{2r}$. Let $\cL(\Lambda_V)_0$ be the subspace of those functions supported on $\Lambda_V^\vee\otimes_{O_E}O_E^{2r}$. Then evaluating at zero induces an isomorphism $\cL(\Lambda_V)_0\simeq\cV_{\widetilde{W}_r,\widetilde{V}}$ of representations of $K_r$. By a deep result of Waldspurger \cite{Wal90}*{Corollaire~III.2}, we have $\cL(\Lambda_V)[\eta]=\omega_{W_r,V}(H_V(F))\cdot\cL(\Lambda_V)_0[\eta]$. After convolution with $\CF_{I_r}$, we have
\begin{align}\label{eq:invariant}
\cL(\Lambda_V)^{I_r}[\kappa_r^-]=\omega_{W_r,V}(H_V(F))\cdot\cL(\Lambda_V)_0^{I_r}[\kappa_r^-].
\end{align}
Let $l_0\in\cL(\Lambda_V)$ be the element that corresponds to $\CF_{\Lambda_V^r}$ in the Schr\"{o}dinger model. It is easy to check that $l_0$ has support in $\Lambda_V^\vee\otimes_{O_E}O_E^{2r}$. Then, by Lemma \ref{le:generator}, we have $l_0\in\cL(\Lambda_V)_0^{I_r}[\kappa_r^-]$. Thus, the proposition follows from the claim (a) and \eqref{eq:invariant}.

For the claim (a), under the Schr\"{o}dinger model $C^\infty(\widetilde{V}^r)$, it is easy to see that a function that is invariant under $I_r$ has to be supported on the zero element. Moreover, such function is an eigenvector of $\dC[I_r\backslash K_r/I_r]$ with respect to the character $\kappa_r^-$. The claim (a) follows.
\end{proof}

The following proposition improves \cite{Ral82}*{Proposition~2.2} in the unitary case, proved by a refinement of Rallis' argument, which answers a question in \cite{Ral82}*{Remark~2.3} for the hermitian Siegel parahoric subgroup.

\begin{proposition}\label{pr:vanishing}
Let $\phi$ be an element in $C^\infty_c(V^r)^{I_r^0}$ such that $\omega_{W_r,V}(g)\phi$ vanishes on $\Sigma_r(V)$ for every $g\in G_r(F)$. Then $\phi=0$.
\end{proposition}

\begin{proof}
We prove by induction on $r$ that
\begin{itemize}
  \item [(a)] If $\phi\in C^\infty_c(V^r)^{I_r^0}$ satisfies that $\omega_{W_r,V}(w)\phi$ vanishes on $\Sigma_r(V)$ for every $w\in\{\pm1\}^r\subseteq\fW_r\subseteq K_r\subseteq G_r(F)$, then $\phi=0$.
\end{itemize}

The case for $r=0$ is trivial. For a subgroup $\Gamma$ of $G_r(F)$, we denote by $C^\infty_c(V^r)_\Gamma$ the $\Gamma$-invariant quotient of $C^\infty_c(V^r)$. For $0\leq t\leq r$, let $\Sigma_r^t(V)$ be the subset of $x\in V^r$ such that $T(x)_{ij}=0$ if $\max\{i,j\}>t$ (so in particular, $\Sigma_r^0(V)=\Sigma_r(V)$). By the same argument for \cite{Ral82}*{Lemma~2.1}, an element $\phi\in C^\infty_c(V^r)$ equals zero in $C^\infty_c(V^r)_{N_r^t(F)\cap N_r^0(F)}$ if and only if $\phi\res_{\Sigma_r^t(V)}=0$. In particular, for $\phi$ in (a), we have $\omega_{W_r,V}(w)\phi=0$ in $C^\infty_c(V^r)_{N_0^t(F)}$ for every $w\in\{\pm1\}^r$. We further claim that
\begin{itemize}
  \item [(b)] For $\phi$ in (a), we have $\omega_{W_r,V}(w)\phi=0$ in $C^\infty_c(V^r)_{N_r^t(F)}$ for every $w\in\{\pm1\}^r$ and every $1\leq t\leq r-1$.
\end{itemize}
Assuming (b), we prove (a). By the Iwasawa decomposition $G_r(F)=\coprod_{w\in\{\pm1\}^r}P_r(F)w I_r^0$ and the fact that $P_r(F)$ normalizes $N_r^t(F)$, we have $\omega_{W_r,V}(g)\phi=0$ in $C^\infty_c(V^r)_{N_r^t(F)}$ for every $g\in G_r(F)$ and every $0\leq t\leq r-1$. In particular, the subrepresentation $\cV_\phi\subseteq C^\infty_c(V^r)$ of $G_r(F)$ generated by $\phi$ is quasi-cuspidal. Since $\cV_\phi$ is finitely generated, it is admissible, hence supercuspidal. Thus, we have $\cV_\phi^{I_r^0}=0$, which implies $\phi=0$.

It remains to prove (b). Take $w$ and $t$ as in (b). Let $\tilde{P}_r^t$ be the subgroup of $P_r^t$ whose Levi factor belongs to $G_t\times\{1_{r-t}\}$. We regard $C^\infty_c(V^r)$ as $C^\infty_c(V^{r-t},C^\infty_c(V^t))$ by identifying $\phi$ as the assignment sending $\zeta=(\zeta_{t+1},\dots,\zeta_r)\in V^{r-t}$ to the function $\phi_\zeta\in C^\infty_c(V^t)$ defined by the formula $\phi_\zeta(x)=\phi(x_1,\dots,x_t,\zeta_{t+1},\dots,\zeta_r)$ for $x=(x_1,\dots,x_t)\in V^t$. For every $\zeta\in V^{r-t}$, there is a unique representation $\omega_{W_r,V}^\zeta$ of $\tilde{P}_r^t(F)$ on $C^\infty_c(V^t)$ such that for every $g\in\tilde{P}_r^t(F)$, the following diagram
\[
\xymatrix{
C^\infty_c(V^r) \ar[r]^-{-_\zeta}\ar[d]_-{\omega_{W_r,V}(g)} & C^\infty_c(V^t) \ar[d]^-{\omega_{W_r,V}^\zeta(g)} \\
C^\infty_c(V^r) \ar[r]^-{-_\zeta} & C^\infty_c(V^t)
}
\]
commutes. For the claim $\omega_{W_r,V}(w)\phi=0$ in $C^\infty_c(V^r)_{N_r^t(F)}$ in (b), it suffices to show that for $\phi$ in (a), the image of $\omega_{W_r,V}(w)\phi$ in $C^\infty_c(V^{r-t},C^\infty_c(V^t)_{N^t_r(F)})$ is zero. There are two cases.

Suppose that $\zeta\not\in\Sigma_{r-t}(V)$. Then we already know that $(\omega_{W_r,V}(w)\phi)_\zeta=0$ in $C^\infty_c(V^t)_{N^t_r(F)\cap N^0_r(F)}$, hence in $C^\infty_c(V^t)_{N^t_r(F)}$.

Suppose that $\zeta\in\Sigma_{r-t}(V)$. Let $Y_\zeta$ be the $E$-subspace of $V$ spanned by $\zeta_{t+1},\dots,\zeta_r$, and put $V_\zeta\coloneqq Y_\zeta^\perp/Y_\zeta$ which is a hermitian space of even dimension. It is clear that as a representation of $G_t(F)$, the Jacquet module $C^\infty_c(V^t)_{N^t_r(F)}$ is isomorphic to the Weil representation $\omega_{W_t,V_\zeta}$ (restricted to $G_t(F)$). Let $\tilde{w}\in\{\pm 1\}^r$ be the unique element that acts trivially on $G_t(F)$ and such that $w\tilde{w}^{-1}\in G_t(F)$. Then $(\omega_{W_r,V}(\tilde{w})\phi)_\zeta$, regarded as an element in $C^\infty_c(V^t)_{N^t_r(F)}$, satisfies the prerequisites in (a) for $\omega_{W_t,V_\zeta}$. Now by the induction hypothesis, we have $(\omega_{W_r,V}(\tilde{w})\phi)_\zeta=0$ in $C^\infty_c(V^t)_{N^t_r(F)}$. As the element $w\tilde{w}^{-1}$ normalizes $N^t_r(F)$, we have $(\omega_{W_r,V}(w)\phi)_\zeta=0$ in $C^\infty_c(V^t)_{N^t_r(F)}$ as well.

Therefore, (b), hence (a), hence the proposition are all proved.
\end{proof}

\section{Annihilator of the spherical module}
\label{ss:proof}

In this section, we compute the annihilator of the spherical module $\cS_{W_r,V}$. We follow the approach used by Rallis in \cite{Ral82}. We first recall the construction of unramified principal series.

We have the canonical isomorphism $M_r=(\Res_{E/F}\GL_1)^r$, under which we write an element of $M_r(F)$ as $a=(a_1,\dots,a_r)$ with $a_i\in E^\times$ its eigenvalue on $e_i$ for $1\leq i\leq r$. For every tuple $\sigma=(\sigma_1,\dots,\sigma_r)\in(\dC/\tfrac{\pi i}{\log q}\dZ)^r$, we define a character $\chi^\sigma_r$ of $M_r(F)$, hence of $P_r(F)$, by the formula
\[
\chi^\sigma_r(a)=\prod_{i=1}^r|a_i|_E^{\sigma_i+i-1/2}.
\]
We then have the normalized principal series
\[
\rI^\sigma_{W_r}\coloneqq\{\varphi\in C^\infty(G_r(F))\res\varphi(ag)=\chi^\sigma_r(a)\varphi(g)\text{ for $a\in P_r(F)$ and $g\in G_r(F)$}\},
\]
which is an admissible representation of $G_r(F)$ via the right translation.

Let $(V,(\;,\;)_V)$ be a hermitian space over $E$ of even dimension, with $d=d(V)$, $s=s(V)$, and $\epsilon=\epsilon(V)$. We fix a decomposition
\[
\Lambda_V=O_Ev_s\oplus\cdots\oplus O_Ev_1 \oplus \Lambda_V^0\oplus \fp_E^\fc v_{-1}\oplus\cdots\oplus \fp_E^\fc v_{-s},
\]
satisfying $(v_i,v_j)_V=\delta_{i,-j}$ for every $i,j$ and such that $v_i$ is perpendicular to $\Lambda^0_V$ for every $i$. Put $V^0\coloneqq\Lambda_V^0\otimes_{O_E}E$, which has dimension $1-(\epsilon 1)$ over $E$.

For $1\leq i\leq s$, let $Z_i$ and $Z^-_i$ be the $E$-subspaces of $V$ spanned by $\{v_s,\dots,v_i\}$ and $\{v_{-s},\dots,v_{-i}\}$, respectively. Then we have an increasing filtration
\begin{align}\label{eq:filtration}
\{0\}=Z_{s+1}\subseteq Z_s\subseteq \cdots \subseteq Z_1
\end{align}
of isotropic $E$-subspaces of $V$ in which $Z_i$ has dimension $s+1-i$. Let $Q_V$ be the (minimal) parabolic subgroup of $H_V$ that stabilizes \eqref{eq:filtration}. Let $M_V$ be the Levi factor of $Q_V$ stabilizing $V^0$ and the lines spanned by $v_i$ for every $i$. Then we have the canonical isomorphism
\[
M_V=\rU(V^0)\times(\Res_{E/F}\GL_1)^s,
\]
under which we write an element of $M_V(F)$ as $b=(b_0,b_1,\dots,b_s)$ with $b_0\in\rU(V^0)(F)$ and $b_i\in E^\times$ its eigenvalue on $v_i$ for $1\leq i\leq s$. For every tuple $\sigma=(\sigma_1,\dots,\sigma_s)\in(\dC/\tfrac{\pi i}{\log q}\dZ)^s$, we define a character $\chi^\sigma_V$ of $M_V(F)$, hence of $Q_V(F)$, by the formula
\begin{align}\label{eq:character}
\chi^\sigma_V(b)=\prod_{i=1}^{s}|b_i|_E^{\sigma_i+i-(\epsilon 1)/2}.
\end{align}

We then have the normalized principal series
\[
\rI^\sigma_V\coloneqq\{\varphi\in C^\infty(H_V(F))\res\varphi(bh)=\chi^\sigma_V(b)\varphi(h)\text{ for $b\in Q_V(F)$ and $h\in H_V(F)$}\},
\]
which is an admissible representation of $H_V(F)$ via the right translation.

\begin{notation}\label{no:pi_v}
We denote by $\pi_V^\sigma$ the unique irreducible constituent of $\rI^\sigma_V$ that has nonzero $L_V$-invariants.\footnote{Recall that the uniqueness is a consequence of the Iwasawa decomposition (see, for example, \cite{Car79}*{Section~3.5}).}
\end{notation}

Now we construct certain intertwining maps from $C^\infty_c(V^r)$ to principal series, which are first constructed by Rallis in the symplectic-orthogonal case \cite{Ral82}. Let $Q^r_V$ be the subgroup of $Q_V$ that preserves both $Z_{s+1-\min\{r,s\}}$ and $Z^-_{s+1-\min\{r,s\}}$, and acts trivially on $(Z_{s+1-\min\{r,s\}}\oplus Z^-_{s+1-\min\{r,s\}})^\perp$. Fix a left invariant Haar measure $\r{d}b$ on $Q^r_V(F)$. For every $\sigma=(\sigma_{s+1-\min\{r,s\}},\dots,\sigma_s)\in(\dC/\tfrac{\pi i}{\log q}\dZ)^{\min\{r,s\}}$ and $\phi\in C^\infty_c(V^r)$, we consider the integral
\[
Z_\sigma(\phi)\coloneqq\int_{Q^r_V(F)}\phi(bv_{s+1-\min\{r,s\}},\dots,bv_s,0,\dots,0)\chi^\sigma_V(b)\rd b,
\]
where $\chi^\sigma_V(b)$ is defined by the same formula \eqref{eq:character} as $b_1=\cdots=b_{s-\min\{r,s\}}=1$ for $b\in Q^r_V(F)$. Since $\phi$ is compactly supported, when $\RE \sigma_i\gg0$ for $s+1-\min\{r,s\}\leq i\leq s$, the integral $Z_\sigma(\phi)$ is absolutely convergent for all $\phi\in C^\infty_c(V^r)$. In this case, we obtain a $G_r(F)\times H_V(F)$-intertwining map
\begin{align}\label{eq:intertwining}
\cZ_\sigma\colon C^\infty_c(V^r)\to C^\infty(G_r(F)\times H_V(F))
\end{align}
that sends $\phi$ to the function $(g,h)\mapsto Z_\sigma(\omega_{W_r,V}(g,h)\phi)$ on $G_r(F)\times H_V(F)$. The following lemma generalizes \cite{Liu11}*{Lemma~A.3}.

\begin{lem}\label{le:principal}
For every $\sigma=(\sigma_{s+1-\min\{r,s\}},\dots,\sigma_s)\in(\dC/\tfrac{\pi i}{\log q}\dZ)^{\min\{r,s\}}$ with $\RE \sigma_i\gg0$ for $s+1-\min\{r,s\}\leq i\leq s$, the image of the intertwining map $\cZ_\sigma$ \eqref{eq:intertwining} is contained in the subspace $\rI^{\overleftarrow\sigma}_{W_r}\boxtimes\rI^{\overrightarrow\sigma}_V$ of $C^\infty(G_r(F)\times H_V(F))$, where
\begin{align*}
\begin{dcases}
\overleftarrow\sigma\coloneqq(-\sigma,-\tfrac{(\epsilon 1)}{2},-\tfrac{2+(\epsilon1)}{2},\dots,-\tfrac{2(r-\min\{r,s\}-1)+(\epsilon 1)}{2})\in(\dC/\tfrac{\pi i}{\log q}\dZ)^r,\\
\overrightarrow\sigma\coloneqq(-\tfrac{2-(\epsilon 1)}{2},-\tfrac{4-(\epsilon 1)}{2},\dots,-\tfrac{2(s-\min\{r,s\})-(\epsilon 1)}{2},\sigma)\in(\dC/\tfrac{\pi i}{\log q}\dZ)^s.
\end{dcases}
\end{align*}
\end{lem}

\begin{proof}
It amounts to showing that for every $\phi\in C^\infty_c(V^r)$, we have
\begin{align}\label{eq:intertwining2}
Z_\sigma(\omega_{W_r,V}(a)\phi)=\chi_r^{\overleftarrow\sigma}(a)Z_\sigma(\phi),\qquad\forall a\in P_r(F),
\end{align}
and
\begin{align}\label{eq:intertwining1}
Z_\sigma(\omega_{W_r,V}(b)\phi)=\chi_V^{\overrightarrow\sigma}(b)Z_\sigma(\phi),\qquad\forall b\in Q_V(F).
\end{align}

For \eqref{eq:intertwining2}, we may assume that $a$ is the diagonal matrix with components $(a_1,\dots,a_r)$ as $Z_\sigma(\omega_{W_r,V}(a)\phi)=Z_\sigma(\phi)$ if $a\in P_r(F)$ is unipotent. Then we have
\begin{align*}
Z_\sigma(\omega_{W_r,V}(a)\phi)
&=|a_1\cdots a_r|_E^d
\int_{Q^r_V(F)}\phi(bv_{s+1-\min\{r,s\}}a_1,\dots,bv_sa_{\min\{r,s\}},0,\dots,0)\chi^\sigma_V(b)\rd b \\
&=|a_1\cdots a_r|_E^{s+1/2-(\epsilon1)/2}\chi^\sigma_V(a_1,\dots,a_{\min\{r,s\}})^{-1}\prod_{i=1}^{\min\{r,s\}}|a_i|^{2i-\min\{r,s\}-1} \\
&\qquad\times\int_{Q^r_V(F)}\phi(bv_{s+1-\min\{r,s\}},\dots,bv_s,0,\dots,0)\chi^\sigma_V(b)\rd b \\
&=\prod_{i=1}^{\min\{r,s\}}|a_i|_E^{s+1/2-(\epsilon1)/2-(\sigma_{i+s-\min\{r,s\}}+i+s-\min\{r,s\}-(\epsilon1)/2)+2i-\min\{r,s\}-1} \\
&\qquad\times\prod_{i=\min\{r,s\}+1}^r|a_i|_E^{s+1/2-(\epsilon1)/2} \times Z_\sigma(\phi)\\
&=\chi_r^{\overleftarrow\sigma}(a)Z_\sigma(\phi).
\end{align*}

For \eqref{eq:intertwining1}, let $\tilde{b}\in Q_V^r(F)$ be the image of $b$ under the canonical quotient map $Q_V\to Q_V^r$. Then we have
\begin{align*}
Z_\sigma(\omega_{W_r,V}(b)\phi)
&=\int_{Q^r_V(F)}\phi(\tilde{b}^{-1}b'v_{s+1-\min\{r,s\}},\dots,\tilde{b}^{-1}b'v_s,0,\dots,0)\chi^\sigma_V(b')\rd b' \\
&=\chi^\sigma_V(\tilde{b})\int_{Q^r_V(F)}\phi(\tilde{b}^{-1}b'v_{s+1-\min\{r,s\}},\dots,\tilde{b}^{-1}b'v_s,0,\dots,0)
\chi^\sigma_V(\tilde{b}^{-1}b')\rd b' \\
&=\chi^\sigma_V(\tilde{b})\int_{Q^r_V(F)}\phi(b'v_{s+1-\min\{r,s\}},\dots,v_s,0,\dots,0)\chi^\sigma_V(b')\rd b' \\
&=\chi^{\overrightarrow\sigma}_V(b)\int_{Q^r_V(F)}\phi(b'v_{s+1-\min\{r,s\}},\dots,v_s,0,\dots,0)\chi^\sigma_V(b')\rd b' \\
&=\chi_V^{\overrightarrow\sigma}(b)Z_\sigma(\phi).
\end{align*}

The lemma follows.
\end{proof}

The following lemma generalizes \cite{Liu11}*{Lemma~A.4}.

\begin{lem}\label{le:vanishing}
Let $\phi$ be an element of $\Sph(V^r)$ such that $Z_\sigma(\phi)=0$ for every element $\sigma=(\sigma_{s+1-\min\{r,s\}},\dots,\sigma_s)\in(\dC/\tfrac{\pi i}{\log q}\dZ)^{\min\{r,s\}}$ with $\RE \sigma_i\gg0$ for $s+1-\min\{r,s\}\leq i\leq s$. Then $\phi=0$.
\end{lem}

\begin{proof}
By the definition of $\kappa_r^\epsilon$, we have $\Sph(V^r)\subseteq C^\infty_c(V^r)^{I_r^0\times L_V}$. As $M_r^0(F)\cap I_r^0=m(\GL_r(O_E))$ under the isomorphism $m\colon\Res_{E/F}\GL_r\xrightarrow{\sim}M_r^0$ in \eqref{eq:levi}, we have
\begin{align}\label{eq:vanishing1}
\phi\res_{Z_{s+1-\min\{r,s\}}^{\min\{r,s\}}\times\{0\}^{s-\min\{r,s\}}}=0
\end{align}
by \cite{Ral82}*{Lemma~5.2} (applied to $M_{\min\{r,s\},\min\{r,s\}}(E)$).

Recall the subset $\Sigma_r(V)\coloneqq\{x\in V^r\res T(x)=0\}$ of $V^r$ from Section \ref{ss:weil}. We claim that
\begin{itemize}
  \item [(a)] $\omega_{W_r,V}(a,h)\phi$ vanishes on $\Sigma_r(V)$ for every $(a,h)\in P_r^0(F)\times H_V(F)$.
\end{itemize}
Let $\Sigma_r^\circ(V)\subseteq\Sigma_r(V)$ be the subset of those $x$ that generates a subspace of $V$ over $E$ of dimension $\min\{r,s\}$ -- the largest possible dimension. As $\Sigma_r^\circ(V)$ is dense in $\Sigma_r(V)$, it suffices to show that $\omega_{W_r,V}(a,h)\phi$ vanishes on $\Sigma_r^\circ(V)$ for every $(a,h)\in M_r^0(F)\times H_V(F)$. By the same argument of \cite{Ral82}*{Lemma~3.1}, $\Sigma_r^\circ(V)$ is transitive under the action of $M_r^0(F)\times H_V(F)$. Thus, it suffices to show that
\begin{align}\label{eq:vanishing3}
\omega_{W_r,V}(a,h)\phi(v_{s+1-\min\{r,s\}},\dots,v_s,0,\dots,0)=0
\end{align}
for every $(a,h)\in M_r^0(F)\times H_V(F)$. Since $\phi$ is invariant under $(M_r^0(F)\cap I_r^0)\times L_V$, by the Iwasawa decomposition, it suffices to consider the case where $a$ is lower-triangular and $h\in Q_V(F)$. Then we have
\[
\omega_{W_r,V}(a,h)\phi(v_{s+1-\min\{r,s\}},\dots,v_s,0,\dots,0)=\phi(x)
\]
for some $x\in Z_{s+1-\min\{r,s\}}^{\min\{r,s\}}\times\{0\}^{s-\min\{r,s\}}$, which is zero by \eqref{eq:vanishing1}. We obtain \eqref{eq:vanishing3}, hence the claim (a).

Now for every $w\in\{\pm 1\}^r\subseteq\fW_r$, there exists an integer $\ell(w)\geq 0$ such that $\omega_{W_r,V}(\CF_{I_rwI_r})\phi$ and $q^{\ell(w)}\omega_{W_r,V}(w)\phi$ coincides on $\Sigma_r(V)$. As $\omega_{W_r,V}(\CF_{I_rwI_r})\phi$ is a multiple of $\phi$, we have that $\omega_{W_r,V}(w)\phi$ vanishes on $\Sigma_r(V)$ by the claim (a). As $G_r(F)=\bigcup_{w\in\{\pm 1\}^r}P_r^0(F)w I_r^0$, we know that $\omega_{W_r,V}(g)\phi$ vanishes on $\Sigma_r(V)$ for every $g\in G_r(F)$. The lemma follows from Proposition \ref{pr:vanishing}.
\end{proof}

\begin{proposition}\label{pr:annihilator}
The annihilator of the $\cH_{W_r}^\epsilon\otimes\cH_V$-module $\cS_{W_r,V}$ is $\cI_{W_r,V}$.
\end{proposition}

\begin{proof}
It suffices to consider the $\cH_{W_r}^\epsilon\otimes\cH_V$-module $\Sph(V^r)$. Let $\cI'_{W_r,V}\subseteq\cH_{W_r}^\epsilon\otimes\cH_V$ be the annihilator of $\Sph(V^r)$. By Definition \ref{de:phi}, Lemma \ref{le:principal}, and Lemma \ref{le:vanishing}, we have $\cI_{W_r,V}\subseteq\cI'_{W_r,V}$. In particular, we may regard $\Sph(V^r)$ as a sheaf over $\Spec(\cH_{W_r}^\epsilon\otimes\cH_V)/\cI_{W_r,V}$. Since $\Spec(\cH_{W_r}^\epsilon\otimes\cH_V)/\cI_{W_r,V}$ is smooth by Remark \ref{re:phi}, if $\cI_{W_r,V}\neq\cI'_{W_r,V}$, then the support of $\Sph(V^r)$ is contained in a proper Zariski closed subset. In other words, for every $\phi\in\Sph(V^r)$, $\cZ_\sigma(\phi)$ vanishes outside a subset of measure zero of those $\sigma=(\sigma_{s+1-\min\{r,s\}},\dots,\sigma_s)\in(\dC/\tfrac{\pi i}{\log q}\dZ)^{\min\{r,s\}}$ with $\RE \sigma_i\gg0$ for $s+1-\min\{r,s\}\leq i\leq s$. However, this is already not the case for $Z_\sigma(\CF_{\Lambda_V^r})$, where $\CF_{\Lambda_V^r}\in\Sph(V^r)$ by Lemma \ref{le:generator}. Thus, we have $\cI_{W_r,V}=\cI'_{W_r,V}$; and the proposition follows.
\end{proof}

Using the map $\cZ_\sigma$ \eqref{eq:intertwining}, we can construct a nontrivial $G_r(F)\times H_V(F)$-intertwining map
from $C^\infty_c(V^r)$ to $\rI^{\overleftarrow\sigma}_{W_r}\boxtimes\rI^{\overrightarrow\sigma}_V$ for every $\sigma\in(\dC/\tfrac{\pi i}{\log q}\dZ)^{\min\{r,s\}}$. However, it is not clear what the maximal semisimple quotient of the image is. To compute the theta lifting of the particular constituent of $\rI^{\overleftarrow\sigma}_{W_r}$ we are interested in, we need to analyze the behavior of the function $\CF_{\Lambda_V^r}$ via doubling zeta integral, which we will do in the next section.

\section{Doubling zeta integrals and doubling L-factors}
\label{ss:doubling}

In this section, we compute certain doubling zeta integrals and doubling $L$-factors for almost unramified representations of $G_r(F)$ (with respect to $K_r$). We will denote by $s$ a general complex variable rather than the Witt index. In fact, the hermitian space $V$ will not appear in this section. We fix the Haar measure on $G_r(F)$ under which $K_r$ has volume $1$.

We briefly review the general theory of doubling zeta integrals \cites{GPR,HKS96,Yam14} for unitary groups. We will follow the reference \cite{Yam14} as we need to use several results from there later, but with slightly modified notation. We apply the setup in \cite{Yam14}*{Section~2} to our skew-hermitian space $W_r$; in particular, we are in the case $(\rI_2)$. We have the doubling skew-hermitian space $W_r^\Box\coloneqq W_r\oplus\bar{W}_r$, where $\bar{W}_r=E^{2r}$ but with the skew-hermitian form given by the matrix $-\(\begin{smallmatrix}&1_r\\ -1_r &\end{smallmatrix}\)$. Similar to $W_r$, we have a natural basis $\{\bar{e}_1,\dots,\bar{e}_{2r}\}$ and a polarization $\bar{W}_r=\bar{X}_r\oplus\bar{Y}_r$. Let $G_r^\Box$ be the unitary group of $W_r^\Box$, which contains $G_r\times G_r$ as a subgroup after we identify $G_r$ with the unitary group of $\bar{W}_r$ in the natural way.

We now take a basis $\{e^\Box_1,\dots,e^\Box_{4r}\}$ of $W_r^\Box$ by the formula
\[
e^\Box_i=e_i,\quad e^\Box_{r+i}=-\bar{e}_i,\quad e^\Box_{2r+i}=e_{r+i}, \quad e^\Box_{3r+i}=\bar{e}_{r+i}
\]
for $1\leq i\leq r$, under which we may identify $W_r^\Box$ with $W_{2r}$ and $G_r^\Box$ with $G_{2r}$. Then we put $P_r^\Box\coloneqq P_{2r}^0$, $N_r^\Box\coloneqq N_{2r}^0$, $K_r^\Box\coloneqq K_{2r}$, $I_r^\Box\coloneqq I_{2r}^0$, and $\fW_r^\Box\coloneqq\fW_{2r}$. We denote by $\Delta\colon P_r^\Box\to\Res_{E/F}\GL_1$ the composition of the Levi quotient map $P_r^\Box=P_{2r}^0\to M_{2r}^0$, the isomorphism $m^{-1}\colon M_{2r}^0\to\Res_{E/F}\GL_{2r}$, and the determinant $\Res_{E/F}\GL_{2r}\to\Res_{E/F}\GL_1$. Put
\[
\bw_r\coloneqq
\begin{pmatrix}
     &  & 1_r &  \\
     & 1_r &  &  \\
    -1_r & 1_r &  &  \\
     &  & 1_r & 1_r \\
\end{pmatrix}
\in G_r^\Box(F).
\]
Then $P_r^\Box\bw_r(G_r\times G_r)$ is Zariski open in $G_r^\Box$.

For $s\in\dC$, the degenerate principal series of $G_r^\Box(F)$ is defined as the normalized induced representation
\[
\rI^\Box_r(s)\coloneqq\Ind_{P_r^\Box}^{G_r^\Box}(|\;|_E^s\circ\Delta)
\]
of $G_r^\Box(F)$. We have the Bruhat decomposition $K_r^\Box=\coprod_{i=0}^{2r}\cB_i^\Box$ into bi-$I_r^\Box$-invariant subsets such that $\cB_i^\Box\prec\cB_j^\Box$ in the Bruhat order for $0\leq i<j\leq 2r$. For $\epsilon=\pm$, let $\ff_{r,\epsilon}^{(s)}$ be the unique section of $\rI^\Box_r(s)$ such that for every $g\in p\cB_i^\Box$ with $p\in P_r^\Box(F)$,
\[
\ff_{r,\epsilon}^{(s)}(g)=(-q)^{\frac{((\epsilon1)-1)i}{2}}\cdot|\Delta(p)|_E^{s+r}.
\]
It is a holomorphic standard, hence good, section.

\begin{remark}\label{re:basis}
Since we have identified $W_r^\Box$ with $W_{2r}$ via the basis $\{e^\Box_1,\dots,e^\Box_{4r}\}$, for every element $\sigma^\Box\in(\dC/\tfrac{\pi i}{\log q}\dZ)^{2r}$, we have the unramified principal series $\rI^{\sigma^\Box}_{W_{2r}}$ of $G_r^\Box(F)$ from Section \ref{ss:proof}. Let $\varphi^{\sigma^\Box}_\epsilon$ be the unique section in $(\rI^{\sigma^\Box}_{W_{2r}})^{I_{2r}}[\kappa_{2r}^\epsilon]$ satisfying $\varphi^{\sigma^\Box}_\epsilon(1_{4r})=1$. By definition, we have $\rI^\Box_r(s)\subseteq\rI^{\sigma^\Box_s}_{W_{2r}}$, where
\[
\sigma^\Box_s\coloneqq(s+r-\tfrac{1}{2},s+r-\tfrac{3}{2},\dots,s-r+\tfrac{3}{2},s-r+\tfrac{1}{2})\in(\dC/\tfrac{\pi i}{\log q}\dZ)^{2r}.
\]
Moreover, by Lemma \ref{le:eigenvector} and Remark \ref{re:eigenvector}, we have $\ff_{r,\epsilon}^{(s)}=\varphi^{\sigma^\Box_s}_\epsilon$. Here, we suppressed $r$ from the notation $\sigma^\Box_s$ as it will be clear at the time we invoke this remark.
\end{remark}

Let $\pi$ be an irreducible admissible representation of $G_r(F)$. For every element $\xi\in\pi^\vee\boxtimes\pi$, we denote by $H_\xi\in C^\infty(G_r(F))$ its associated matrix coefficient. Then for every meromorphic section $f^{(s)}$ of $\rI^\Box_r(s)$, we have the (doubling) zeta integral:
\[
Z(\xi,f^{(s)})\coloneqq\int_{G_r(F)}H_\xi(g)f^{(s)}(\bw_r(g,1_{2r}))\rd g,
\]
which is absolutely convergent for $\RE s$ large enough and has a meromorphic continuation. We let $L(s,\pi)$ and $\varepsilon(s,\pi,\psi_F)$ be the doubling $L$-factor and the doubling epsilon factor of $\pi$, respectively, defined in \cite{Yam14}*{Theorem~5.2}.

\begin{remark}
In \cite{Yam14}, the author considers the degenerate principal series with respect to another parabolic subgroup $P_r^\triangle$, namely, the one stabilizing the diagonal of $W_r^\Box$. Since we have $\bw_r P_r^\triangle \bw_r^{-1}=P_r^\Box$, the assignment that sends $f^{(s)}$ to the function $g\mapsto f^{(s)}(\bw_rg)$ gives rise to an intertwining isomorphism from our degenerate principal series to the one in \cite{Yam14}.
\end{remark}

\begin{definition}\label{de:representation}
Let $\pi$ be an irreducible admissible representation of $G_r(F)$. We say that
\begin{enumerate}
  \item $\pi$ is \emph{unramified} (with respect to $K_r$) if $\pi^{I_r}[\kappa_r^+]\neq\{0\}$; and

  \item $\pi$ is \emph{almost unramified} (with respect to $K_r$) if $\pi^{I_r}[\kappa_r^-]\neq\{0\}$ and that the Satake parameter of $\pi$ contains either $q$ or $q^{-1}$.
\end{enumerate}
\end{definition}

\begin{remark}
Note that $\pi^{I_r}[\kappa_r^+]$ is simply $\pi^{K_r}$ by Lemma \ref{le:eigenvector}. In general, if $\pi^{I_r}\neq\{0\}$, then there exists an element $\sigma\in(\dC/\tfrac{\pi i}{\log q}\dZ)^r$, unique up to the natural action of $\fW_r$, such that $\pi$ is a constituent of $\rI^\sigma_{W_r}$. Then the Satake parameter of $\pi$ contains either $q$ or $q^{-1}$ if and only if the corresponding $\sigma$ contains either $\tfrac{1}{2}$ or $-\tfrac{1}{2}$.
\end{remark}

We now study certain zeta integrals of unramified and almost unramified representations.

\begin{notation}\label{no:pi_w}
For $\sigma=(\sigma_1,\dots,\sigma_r)\in(\dC/\tfrac{\pi i}{\log q}\dZ)^r$, we denote by $\pi^\sigma_{W_r,\epsilon}$ the unique irreducible constituent of $\rI^\sigma_{W_r}$ satisfying $(\pi^\sigma_{W_r,\epsilon})^{I_r}[\kappa_r^\epsilon]\neq\{0\}$.
\end{notation}

For example, $\pi^{(1/2)}_{W_1,+}$ and $\pi^{(1/2)}_{W_1,-}$ are simply the trivial character and the Steinberg representation of $G_1(F)$, respectively. We have that $(\pi^\sigma_{W_r,\epsilon})^\vee$ is isomorphic to $\pi^\sigma_{W_r,\epsilon}$, so that $((\pi^\sigma_{W_r,\epsilon})^\vee)^{I_r}[\kappa_r^\epsilon]\boxtimes(\pi^\sigma_{W_r,\epsilon})^{I_r}[\kappa_r^\epsilon]$ is one dimensional. We let $\xi^\sigma_\epsilon$ be a generator of this one dimensional space; it satisfies $H_{\xi^\sigma_\epsilon}(1_{2r})\neq 0$. We normalize $\xi^\sigma_\epsilon$ so that $H_{\xi^\sigma_\epsilon}(1_{2r})=1$, which makes it unique. For future use, we introduce the following $L$-factors:
\begin{align}\label{eq:factor}
\begin{dcases}
a_{2r}(s)\coloneqq\prod_{i=1}^{2r}\frac{1}{1-(-1)^iq^{-2s+i-1}},\\
b_{2r}(s)\coloneqq\prod_{i=1}^{2r}\frac{1}{1-(-1)^iq^{-2s-i}},\\
L^\sigma_+(s)\coloneqq\prod_{i=1}^r\frac{1}{(1-q^{2\sigma_i-2s})(1-q^{-2\sigma_i-2s})},\\
L^\sigma_-(s)\coloneqq(1-q^{1-2s})\prod_{i=1}^r\frac{1}{(1-q^{2\sigma_i-2s})(1-q^{-2\sigma_i-2s})},\\
c^r_+(s)\coloneqq 1,\\
c^r_-(s)\coloneqq\frac{(1+q)}{(-q)^rq(1+q^{2r-1})(1-q^{-2s-2r})}.
\end{dcases}
\end{align}

\begin{proposition}\label{pr:zeta}
For $\sigma\in(\dC/\tfrac{\pi i}{\log q}\dZ)^r$, we have
\[
Z(\xi^\sigma_\epsilon,\ff_{r,\epsilon}^{(s)})=c^r_\epsilon(s)\cdot\frac{L^\sigma_\epsilon(s+\tfrac{1}{2})}{b_{2r}(s)}.
\]
\end{proposition}

\begin{proof}
When $\epsilon=+$, this is proved in \cite{Li92}*{Theorem~3.1}. Now we consider the case where $\epsilon=-$. Put $\pi\coloneqq\pi^\sigma_{W_r,-}$ to ease notation. Recall the Bruhat decomposition $K_r=\coprod_{i=0}^r\cB_i$ from Remark \ref{re:eigenvector}.

Recall the isomorphism $m\colon\Res_{E/F}\GL_r\xrightarrow{\sim}M_r^0$ from \eqref{eq:levi}. Let $\tau$ be the unramified constituent of the normalized induction of $\boxtimes_{i=1}^r|\;|_E^{\sigma_i}$. We fix vectors $v_0\in\tau$ and $v_0^\vee\in\tau^\vee$ fixed by $M_r^0(F)\cap K_r=m(\GL_r(O_E))$ such that $\langle v_0^\vee,v_0\rangle_\tau=1$. Put $\Pi\coloneqq\Ind_{P_r^0}^{G_r}(\tau)$ and identify $\Pi^\vee$ with $\Ind_{P_r^0}^{G_r}(\tau^\vee)$ via the pairing
\[
\langle\varphi^\vee,\varphi\rangle_\Pi\coloneqq\int_{K_r}\langle\varphi^\vee(k),\varphi(k)\rangle_\tau\rd k.
\]
Let $\varphi_0\in\Ind_{P_r^0}^{G_r}(\tau)$ and $\varphi_0^\vee\in\Ind_{P_r^0}^{G_r}(\tau^\vee)$ be the unique vectors such that $\varphi_0(k)=(-q)^{-i}\tau_0$ and $\varphi_0^\vee(k)=(-q)^{-i}\tau_0^\vee$ if $k\in\cB_i$ for $0\leq i\leq r$, respectively. By Lemma \ref{le:eigenvector} and Remark \ref{re:eigenvector}, we have $\varphi_0\in\Pi^{I_r}[\kappa_r^-]$ and $\varphi_0^\vee\in(\Pi^\vee)^{I_r}[\kappa_r^-]$. Thus, $\pi$ and $\pi^\vee$ are constituents of $\Pi$ and $\Pi^\vee$, respectively; and we have
\begin{align}\label{eq:zeta}
Z(\xi^\sigma_-,\ff_{r,-}^{(s)})=C^{-1}
\int_{G_r(F)}\ff_{r,-}^{(s)}(\bw_r(g,1_{2r}))\langle\Pi^\vee(g)\varphi^\vee_0,\varphi_0\rangle_\Pi\rd g
\end{align}
where
\[
C\coloneqq\langle\varphi^\vee_0,\varphi_0\rangle_\Pi=\sum_{i=0}^rq^{-2i}\vol(\cB_i)=\frac{1+q}{q(1+q^{2r-1})}.
\]
To proceed, note that for $k\in K_r$, we have $\bw_r(k,k)\bw_r^{-1}\in I_r^\Box$, hence $|\Delta(\bw_r(k,k)\bw_r^{-1})|_E=1$. Then we have
\begin{align*}
\eqref{eq:zeta}&=C^{-1}
\int_{G_r(F)}\ff_{r,-}^{(s)}(\bw_r(g,1_{2r}))\int_{K_r}\langle\varphi^\vee_0(kg),\varphi_0(k)\rangle_\tau\rd k\rd g \\
&=C^{-1}\int_{G_r(F)}\ff_{r,-}^{(s)}(\bw_r(g,1_{2r}))\sum_{i=0}^r\int_{\cB_i}(-q)^{-i}\langle\varphi^\vee_0(kg),\varphi_0(1_{2r})\rangle_\tau\rd k\rd g \\
&=C^{-1}\int_{G_r(F)}\sum_{i=0}^r(-q)^{-i}\int_{\cB_i}\ff_{r,-}^{(s)}(\bw_r(k^{-1}g,1_{2r}))\rd k
\langle\varphi^\vee_0(g),\varphi_0(1_{2r})\rangle_\tau\rd g\\
&=C^{-1}\int_{G_r(F)}\sum_{i=0}^r(-q)^{-i}\int_{\cB_i}\ff_{r,-}^{(s)}(\bw_r(g,k))\rd k
\langle\varphi^\vee_0(g),\varphi_0(1_{2r})\rangle_\tau\rd g.
\end{align*}
Recall the Iwasawa decomposition $G_r(F)=\coprod_{i=0}^r M_r^0(F)N_r^0(F)\cB_i$, and equip $M_r^0(F)$ and $N_r^0(F)$ with the Haar measures under which $M_r^0(F)\cap K_r$ and $N_r^0(F)\cap K_r$ have volume $1$, respectively. Then we have
\begin{align}\label{eq:zeta1}
&\quad Z(\xi^\sigma_-,\ff_{r,-}^{(s)}) \\
&=C^{-1}\int_{M_r^0(F)}\int_{N_r^0(F)}\int_{K_r}\sum_{i=0}^r(-q)^{-i}\int_{\cB_i}\ff_{r,-}^{(s)}(\bw_r(mnk',k))\rd k
\langle\varphi^\vee_0(mnk'),\varphi_0(1_{2r})\rangle_\tau\rd k'\rd n\rd m \notag\\
&=C^{-1}\int_{M_r^0(F)}\int_{N_r^0(F)}\sum_{j=0}^r\sum_{i=0}^r(-q)^{-i-j}\int_{\cB_j}\int_{\cB_i}\ff_{r,-}^{(s)}(\bw_r(mnk',k))\rd k\rd k'
\langle\varphi^\vee_0(m),\varphi_0(1_{2r})\rangle_\tau\rd n\rd m \notag\\
&=C^{-1}\int_{\GL_r(E)}\(\sum_{j=0}^r\sum_{i=0}^r(-q)^{-i-j}\int_{\cB_j}\int_{\cB_i}F^s_-(m(a)k',k)\rd k\rd k'\)
|\dtm a|_E^{-r/2}\langle\tau^\vee(a)v_0^\vee,v_0\rangle_\tau\rd a, \notag
\end{align}
where
\[
F^s_\epsilon(g)\coloneqq\int_{N_r^0(F)}\varphi^{\sigma^\Box_s}_\epsilon(\bw_r(n,1_{2r})g)\rd n
\]
for $\epsilon=\pm$ and $g\in G_r^\Box(F)$. Here, we have adopted Remark \ref{re:basis} for the last equality.

Let $\bw'_r\in\fW_r^\Box$ be the element that acts by the following way
\[
\bw'_r\colon
\begin{dcases}
e^\Box_i\mapsto e^\Box_{3r+1-i},\\
e^\Box_{r+i}\mapsto e^\Box_{r+i},\\
e^\Box_{2r+i}\mapsto -e^\Box_{r+1-i},\\
e^\Box_{3r+i}\mapsto e^\Box_{3r+i},
\end{dcases}
\]
for $1\leq i\leq r$. Put $\bw''_r=(\bw'_r)^{-1}\bw_r$, which belongs to $N_r^\Box(F)$ and normalizes $N_r^0(F)\times\{1_{2r}\}$ with the Haar measure preserved. Then we have
\begin{align}\label{eq:zeta2}
F^s_\epsilon(g)=T_{\bw'_r}(\varphi^{\sigma^\Box_s}_\epsilon)(\bw''_rg),
\end{align}
where $T_{\bw'_r}$ is the intertwining operator in \cite{Cas80}*{Section~3}. Since $(\rI^{\sigma^\Box}_{W_{2r}})^{I_{2r}}[\kappa_{2r}^\epsilon]$ is spanned by $\varphi^{\sigma^\Box}_\epsilon$, we have
\begin{align}\label{eq:zeta3}
T_{\bw'_r}(\varphi^{\sigma^\Box_s}_\epsilon)=C_{\bw'_r}^\epsilon(s)\cdot\varphi^{\bw'_r\sigma^\Box_s}_\epsilon
\end{align}
for some meromorphic function $C_{\bw'_r}^\epsilon(s)$. By the Gindikin--Karpelevich formula (see \cite{Cas80}*{Theorem~3.1}), we have
\[
C_{\bw'_r}^+(s)=\prod_{i=1}^r\frac{\zeta_E(2s+2i)}{\zeta_E(2s+r+i)}\prod_{i=1}^r\frac{\zeta_F(2s+2i-1)}{\zeta_F(2s+2i)};
\]
and by successively applying \cite{Cas80}*{Theorem~3.4}, we have
\begin{align*}
C_{\bw'_r}^-(s)
&=\prod_{1\leq i<j\leq r}\frac{1-(q^{-2})^{2s+2r-i-j+2}}{1-(q^{-2})^{2s+2r-i-j+1}}
\prod_{i=1}^r\frac{q^{-2s-2r+2i}-1}{q(1-q^{-2s-2r+2i-1})} \\
&=(-q)^{-r}\prod_{1\leq i<j\leq r}\frac{\zeta_E(2s+2r-i-j+1)}{\zeta_E(2s+2r-i-j+2)}
\prod_{i=1}^r\frac{\zeta_F(2s+2r-2i+1)}{\zeta_F(2s+2r-2i)} \\
&=(-q)^{-r}\prod_{i=1}^r\frac{\zeta_E(2s+2i)}{\zeta_E(2s+r+i)}\prod_{i=1}^r\frac{\zeta_F(2s+2i-1)}{\zeta_F(2s+2i-2)}.
\end{align*}
Together, we have
\begin{align}\label{eq:zeta6}
\frac{C_{\bw'_r}^-(s)}{C_{\bw'_r}^+(s)}=(-q)^{-r}\frac{\zeta_F(2s+2r)}{\zeta_F(2s)}.
\end{align}

Combining \eqref{eq:zeta1}, \eqref{eq:zeta2}, and \eqref{eq:zeta3}, we have
\begin{align}\label{eq:zeta4}
& Z(\xi^\sigma_-,\ff_{r,-}^{(s)})=C^{-1}C_{\bw'_r}^-(s)\int_{\GL_r(E)} \\
&\(\sum_{j=0}^r\sum_{i=0}^r(-q)^{-i-j}\int_{\cB_j}\int_{\cB_i}
\varphi^{\bw'_r\sigma^\Box_s}_-(\bw''_r(m(a)k',k))\rd k\rd k'\)
|\dtm a|_E^{-r/2}\langle\tau^\vee(a)v^\vee_0,v_0\rangle_\tau\rd a. \notag
\end{align}
As $\bw''_r(m(a),1_{2r})\in P_r^\Box(F)$, we have
\[
\varphi^{\bw'_r\sigma^\Box_s}_-(\bw''_r(m(a)k',k))=(-q)^{-i-j}\varphi^{\bw'_r\sigma^\Box_s}_-(\bw''_r(m(a),1_{2r}))
\]
if $(k,k')\in\cB_i\times\cB_j\subseteq\cB_{i+j}^\Box$. By \eqref{eq:zeta4}, we have
\begin{align}\label{eq:zeta5}
Z(\xi^\sigma_-,\ff_{r,-}^{(s)})
&=CC_{\bw'_r}^-(s)\int_{\GL_r(E)}\varphi^{\bw'_r\sigma^\Box_s}_-(\bw''_r(m(a),1_{2r}))|\dtm a|_E^{-r/2}\langle\tau^\vee(a)v^\vee_0,v_0\rangle_\tau\rd a \\
&=CC_{\bw'_r}^-(s)\int_{\GL_r(E)}\varphi^{\bw'_r\sigma^\Box_s}_+(\bw''_r(m(a),1_{2r}))|\dtm a|_E^{-r/2}\langle\tau^\vee(a)v^\vee_0,v_0\rangle_\tau\rd a. \notag
\end{align}

By a similar argument, we have
\begin{align}\label{eq:zeta8}
Z(\xi^\sigma_+,\ff_{r,+}^{(s)})=C_{\bw'_r}^+(s)\int_{\GL_r(E)}\varphi^{\bw'_r\sigma^\Box_s}_+(\bw''_r(m(a),1_{2r}))
|\dtm a|_E^{-r/2}\langle\tau^\vee(a)v^\vee_0,v_0\rangle_\tau\rd a.
\end{align}
Combining \eqref{eq:zeta5}, \eqref{eq:zeta8}, and \eqref{eq:zeta6}, we have
\begin{align*}
Z(\xi^\sigma_-,\ff_{r,-}^{(s)})&=C\cdot\frac{C_{\bw'_r}^-(s)}{C_{\bw'_r}^+(s)}\cdot Z(\xi^\sigma_+,\ff_{r,+}^{(s)}) \\
&=(-q)^{-r}C\cdot\frac{\zeta_F(2s+2r)}{\zeta_F(2s)}\cdot Z(\xi^\sigma_+,\ff_{r,+}^{(s)}) \\
&=(-q)^{-r}C\cdot\frac{\zeta_F(2s+2r)}{\zeta_F(2s)}\cdot\frac{L^\sigma_+(s+\tfrac{1}{2})}{b_{2r}(s)} \\
&=c^r_-(s)\cdot\frac{L^\sigma_-(s+\tfrac{1}{2})}{b_{2r}(s)}.
\end{align*}
The proposition is proved.
\end{proof}

Recall from \cite{Yam14}*{Section~3.5} that we have the intertwining operator
\[
M(s)\colon\rI^\Box_r(s)\to\rI^\Box_r(-s),
\]
and its normalized version
\[
M^\dag_{\psi_F}(s)\coloneqq q^{4\fc rs}\frac{b_{2r}(-s)}{a_{2r}(s)}M(s),
\]
where $a_{2r}$ and $b_{2r}$ are those in \eqref{eq:factor}.

\begin{lem}\label{le:intertwining}
For $\epsilon=\pm$, we have
\[
\frac{b_{2r}(s)}{c^r_\epsilon(s)}\cdot M^\dag_{\psi_F}(s)\ff^{(s)}_{r,\epsilon}=\epsilon
q^{(4\fc r+(\epsilon1)-1)s}\cdot\frac{b_{2r}(-s)}{c^r_\epsilon(-s)}\cdot\ff^{(-s)}_{r,\epsilon}.
\]
\end{lem}

\begin{proof}
When $\epsilon=+$, this is proved in \cite{Yam14}*{Proposition~3.1(5)}. Now we consider the case where $\epsilon=-$.

By the relation $\ff_{r,\epsilon}^{(s)}=\varphi^{\sigma^\Box_s}_\epsilon$ from Remark \ref{re:basis} and successively applying \cite{Cas80}*{Theorem~3.4}, we have
\begin{align*}
M(s)\ff^{(s)}_{r,-}&=\prod_{1\leq i<j\leq 2r}\frac{1-(q^{-2})^{2s+2r-i-j+2}}{1-(q^{-2})^{2s+2r-i-j+1}}
\prod_{i=1}^{2r}\frac{q^{-2s-2r+2i}-1}{q(1-q^{-2s-2r+2i-1})}\cdot\ff^{(-s)}_{r,-} \\
&=q^{-2r}\cdot\prod_{1\leq i<j\leq 2r}\frac{\zeta_E(2s+2r-i-j+1)}{\zeta_E(2s+2r-i-j+2)}
\prod_{i=1}^{2r}\frac{\zeta_F(2s+2r-2i+1)}{\zeta_F(2s+2r-2i)}\cdot\ff^{(-s)}_{r,-} \\
&=q^{-2r}\cdot\prod_{j=1}^{2r}\frac{\zeta_E(2s+2r-2j+2)}{\zeta_E(2s+2r-j+1)}
\prod_{i=1}^{2r}\frac{\zeta_F(2s+2r-2i+1)}{\zeta_F(2s+2r-2i)}\cdot\ff^{(-s)}_{r,-} \\
&=q^{-2r}\cdot\frac{a_{2r}(s)}{b_{2r}(s)}\cdot\frac{\zeta_F(2s+2r)}{\zeta_F(2s-2r)}\cdot\ff^{(-s)}_{r,-} \\
&=-q^{-2s}\cdot\frac{a_{2r}(s)}{b_{2r}(s)}\cdot\frac{\zeta_F(2s+2r)}{\zeta_F(-2s+2r)}\cdot\ff^{(-s)}_{r,-} \\
&=-q^{-2s}\cdot\frac{a_{2r}(s)}{b_{2r}(s)}\cdot\frac{c^r_-(s)}{c^r_-(-s)}\cdot\ff^{(-s)}_{r,-}.
\end{align*}
The lemma follows immediately.
\end{proof}

\begin{theorem}\label{th:gcd}
For $\sigma\in(\dC/\tfrac{\pi i}{\log q}\dZ)^r$, we have
\begin{enumerate}
  \item $L(s,\pi^\sigma_{W_r,+})=L^\sigma_+(s)$ and $\varepsilon(s,\pi^\sigma_{W_r,+},\psi_F)=q^{2\fc r(2s-1)}$;

  \item $L(s,\pi^\sigma_{W_r,-})=L^\sigma_-(s)$ and $\varepsilon(s,\pi^\sigma_{W_r,-},\psi_F)=-q^{(2\fc r-1)(2s-1)}$ if $\sigma$ contains either $\tfrac{1}{2}$ or $-\tfrac{1}{2}$.
\end{enumerate}
\end{theorem}

\begin{proof}
Part (1) is proved in \cite{Yam14}*{Proposition~7.1}. Now we consider (2). Put $\pi\coloneqq\pi^\sigma_{W_r,-}$ to ease notation. We first compute $L(s,\pi)$.

Assume first $r=1$, and $\sigma=(\tfrac{1}{2})$ without lost of generality. By \cite{Yam14}*{Lemma~6.1(1)}, we know that $(1-q^{1-2s})(1-q^{-1-2s})L(s,\pi)$ is entire. By \cite{Yam14}*{Lemma~7.2}, we know that $L(s,\pi)$ is holomorphic for $\RE s>0$, which implies that $L(s,\pi)$ is either $1$ or $(1-q^{-1-2s})^{-1}$. By Proposition \ref{pr:zeta}, we have
\[
Z(\xi^{(1/2)}_-,\ff_{1,-}^{(s)})=\frac{1+q^{-2s-1}}{-q^2(1-q^{-1-2s})},
\]
which is not entire. Thus, we have $L(s,\pi)=(1-q^{-1-2s})^{-1}=L^\sigma_-(s)$.

Now suppose that $r>1$, and that $\sigma=(\tfrac{1}{2},\sigma_2,\dots,\sigma_r)$ without lost of generality. For every subset $J\subseteq\{1,\dots,r\}$, denote by $w_J$ the element $\prod_{i\in J}w_i\in\{\pm1\}^r\subseteq\fW_r\subseteq K_r$. Then we have the Iwasawa decomposition $G_r(F)=\coprod_{J\subseteq\{1,\dots,r\}}P_r(F)w_JI_r^0$. The subspace $(\rI_{W_r}^\sigma)^{I_r}[\kappa_r^-]$ is spanned by the function $\varphi$ satisfying that for every $g\in aw_J I_r^0$ with $a\in P_r(F)$, $\varphi(g)=(-q)^{-|J|}\chi^\sigma_r(a)$. It is clear that $\varphi$ belongs to the subspace of $\rI_{W_r}^\sigma$ that is the normalized induction of $\pi_{W_1,-}^{(1/2)}\boxtimes|\;|_E^{\sigma_2}\boxtimes\cdots\boxtimes|\;|_E^{\sigma_r}$. By \cite{Yam14}*{Lemma~6.1(1)} and the formula for $L(s,\pi_{W_1,-}^{(1/2)})$ just obtained above, we know that $L(s,\pi^\sigma_{W_r,-})/L^\sigma_-(s)$ is entire. By Proposition \ref{pr:zeta}, we have
\[
\frac{Z(\xi^\sigma_-,\ff_{r,-}^{(s)})}{L^\sigma_-(s)}=\frac{(1+q)}{(-q)^rq(1+q^{2r-1})}\prod_{i=1}^{2r-1}(1-(-1)^iq^{-2s-i}),
\]
which does not vanish for $\RE s\geq 0$. By \cite{Yam14}*{Lemma~6.1(2)}, we have $L(s,\pi^\sigma_{W_r,-})=L^\sigma_-(s)$.\footnote{In \cite{Yam14}*{Lemma~6.1(2)}, it considers all unitary characters $\chi$. However, in our case where $\chi=1$, it suffices to consider the trivial character alone through the proof.}

Now we compute $\varepsilon(s,\pi,\psi_F)$. By Lemma \ref{le:intertwining} and Proposition \ref{pr:zeta}, we have
\begin{align*}
\frac{Z(\xi^\sigma_-,M^\dag_{\psi_F}(s)\ff_{r,-}^{(s)})}{L(\tfrac{1}{2}-s,\pi)}
&=-q^{4\fc rs-2s}\cdot\frac{b_{2r}(-s)}{c_r^-(-s)}\frac{c_r^-(s)}{b_{2r}(s)}\cdot
\frac{Z(\xi^\sigma_-,\ff_{r,-}^{(-s)})}{L(\tfrac{1}{2}-s,\pi)} \\
&=-q^{4\fc rs-2s}\cdot\frac{c_r^-(s)}{b_{2r}(s)}\cdot
\frac{L^\sigma_-(\tfrac{1}{2}-s)}{L(\tfrac{1}{2}-s,\pi)} \\
&=-q^{4\fc rs-2s}\cdot\frac{c_r^-(s)}{b_{2r}(s)}.
\end{align*}
Similarly, we have
\[
\frac{Z(\xi^\sigma_-,\ff_{r,-}^{(s)})}{L(s+\tfrac{1}{2},\pi)}=
\frac{c_r^-(s)}{b_{2r}(s)}\cdot\frac{L^\sigma_-(s)}{L(s+\tfrac{1}{2},\pi)}=\frac{c_r^-(s)}{b_{2r}(s)}.
\]
Combining these two identities, we have
\[
\frac{Z(\xi^\sigma_-,M^\dag_{\psi_F}(s)\ff_{r,-}^{(s)})}{L(\tfrac{1}{2}-s,\pi)}=
-q^{(2\fc r-1)(2(s+1/2)-1)}\cdot\frac{Z(\xi^\sigma_-,\ff_{r,-}^{(s)})}{L(s+\tfrac{1}{2},\pi)},
\]
which implies $\varepsilon(s,\pi,\psi_F)=-q^{(2\fc r-1)(2s-1)}$ by \cite{Yam14}*{Theorem~5.2}.

Part (2) is proved.
\end{proof}

\begin{remark}
When $\sigma$ contains neither $\tfrac{1}{2}$ nor $-\tfrac{1}{2}$, by \cite{LTXZZ}*{Remark~C.2.4}, $\pi^\sigma_{W_r,-}$ is unramified with respect to either $K_r$ or the other (conjugacy class of) hyperspecial maximal compact subgroup that is not conjugate to $K_r$ in $G_r(F)$. In particular, we have $L(s,\pi^\sigma_{W_r,-})=L^\sigma_+(s)$ and $\varepsilon(s,\pi^\sigma_{W_r,-},\psi_F)=q^{2\fc r(2s-1)}$.
\end{remark}

\section{Consequences for the theta correspondence}
\label{ss:theta}

Let $(V,(\;,\;)_V)$ be a hermitian space over $E$ of even dimension, with $d=d(V)$, $s=s(V)$, and $\epsilon=\epsilon(V)$.

For an irreducible admissible representation $\pi$ of $G_r(F)$, let $\Theta(\pi,V)$ be the $\pi$-isotypic quotient of $\cV_{W_r,V}$, which is an admissible representation of $H_V(F)$, and $\theta(\pi,V)$ its maximal semisimple quotient. By \cite{Wal90} and \cite{GT16}, $\theta(\pi,V)$ is either zero, or an irreducible admissible representation of $H_V(F)$, known as the \emph{theta lifting} of $\pi$ to $V$ (with respect to the additive character $\psi_F$ and the trivial splitting character). In particular, we have two tautological surjective $G_r(F)\times H_V(F)$-intertwining maps
\begin{align}\label{eq:theta}
\Theta_\pi\colon C^\infty_c(V^r)\to\pi\boxtimes\Theta(\pi,V),\qquad
\theta_\pi\colon C^\infty_c(V^r)\to\pi\boxtimes\theta(\pi,V).
\end{align}

Recall that we have identified $W_r^\Box$ with $W_{2r}$ via the basis $\{e^\Box_1,\dots,e^\Box_{4r}\}$ in Section \ref{ss:doubling}. Put $s_0\coloneqq d-r$. Then we have the $G_r^\Box(F)$-intertwining map $C^\infty_c(V^{2r})\to\rI_r^\Box(s_0)$ sending $\Phi$ to $f_\Phi$ defined by the formula $f_\Phi(g)\coloneqq\omega_{W_{2r},V}(g)\Phi(0)$. Let $f^{(s)}_\Phi$ be the standard section of $\rI_r^\Box(s)$ whose value at $s_0$ is $f_\Phi$.

\begin{lem}\label{le:theta}
We have $f^{(s)}_\Phi=\ff^{(s)}_{r,\epsilon}$ if $\Phi=\CF_{\Lambda_V^{2r}}$.
\end{lem}

\begin{proof}
This is a consequence of Lemma \ref{le:eigenvector} and Lemma \ref{le:generator}.
\end{proof}

\begin{theorem}\label{th:theta}
Consider the representation $\pi\coloneqq\pi_{W_r,\epsilon}^{\overleftarrow\sigma}$ of $G_r(F)$ (Notation \ref{no:pi_w}) for an element $\sigma=(\sigma_{s+1-\min\{r,s\}},\dots,\sigma_s)\in(\dC/\tfrac{\pi i}{\log q}\dZ)^{\min\{r,s\}}$.
\begin{enumerate}
  \item We have $\Theta_\pi(\CF_{\Lambda_V^r})\neq 0$, and that $\pi_V^{\overrightarrow\sigma}$ (Notation \ref{no:pi_v}) is a constituent of $\Theta(\pi,V)$;

  \item Assume Assumption \ref{as:invariant}. We have $\theta_\pi(\CF_{\Lambda_V^r})\neq 0$ and $\theta(\pi,V)\simeq\pi_V^{\overrightarrow\sigma}$. Moreover, if $\rI_V^{\overrightarrow\sigma}$ is irreducible, then $\Theta(\pi,V)=\theta(\pi,V)$.
\end{enumerate}
Here,
\begin{align*}
\begin{dcases}
\overleftarrow\sigma\coloneqq(-\sigma,-\tfrac{(\epsilon 1)}{2},-\tfrac{2+(\epsilon1)}{2},\dots,-\tfrac{2(r-\min\{r,s\}-1)+(\epsilon 1)}{2})\in(\dC/\tfrac{\pi i}{\log q}\dZ)^r,\\
\overrightarrow\sigma\coloneqq(-\tfrac{2-(\epsilon 1)}{2},-\tfrac{4-(\epsilon 1)}{2},\dots,-\tfrac{2(s-\min\{r,s\})-(\epsilon 1)}{2},\sigma)\in(\dC/\tfrac{\pi i}{\log q}\dZ)^s,
\end{dcases}
\end{align*}
are as in Lemma \ref{le:principal}.
\end{theorem}

\begin{proof}
In the proof of this proposition, $s$ will also serve as a complex variable like in Section \ref{ss:doubling}, which should be clear from the context.

For $\Phi\in C^\infty_c(V^{2r})$, define a meromorphic section
\[
\tilde{f}^{(s)}_\Phi\coloneqq
\begin{dcases}
f^{(s)}_\Phi, &\text{if }s_0\geq 0, \\
\zeta_E(s-s_0)f^{(s)}_\Phi, &\text{if }s_0< 0.
\end{dcases}
\]
By \cite{Yam14}*{Lemma~8.2}, $\tilde{f}^{(s)}_\Phi$ is good at $s=s_0$. Since $\tilde{f}^{(s)}_\Phi$ is holomorphic at $s\neq s_0$, it is a good section. The assignment
\[
(\xi,\Phi)\mapsto\cZ(\xi,\Phi)\coloneqq
\frac{Z(\xi,\tilde{f}^{(s)}_\Phi)}{L(s+\tfrac{1}{2},\pi)}|_{s=s_0}\in\dC
\]
defines a $G_r(F)\times G_r(F)$-intertwining map $C^\infty_c(V^{2r})_{H_V(F)}\to\pi\boxtimes\pi^\vee$. We claim that
\begin{align}\label{eq:theta1}
\cZ(\xi_r^{\overleftarrow\sigma},\CF_{\Lambda_V^r}\otimes\CF_{\Lambda_V^r})\neq 0.
\end{align}

Assuming the claim, we prove the theorem. By \eqref{eq:theta1} and the seesaw identity, we have $\Theta_\pi(\CF_{\Lambda_V^r})\neq 0$, hence $\Theta(\pi,V)^{L_V}\neq\{0\}$. By Proposition \ref{pr:annihilator}, the ideal $\cI_{W_r,V}$ annihilates $(\pi_{W_r,\epsilon}^{\overleftarrow\sigma})^{I_r}[\kappa_r^\epsilon]\boxtimes\Theta(\pi,V)^{L_V}$, which implies that $\pi_V^{\overrightarrow\sigma}$ is a constituent of $\Theta(\pi,V)$. Thus, (1) is proved.

For (2), we have $\theta(\pi,V)\neq 0$ by (1). Consider the two $H_V(F)$-intertwining maps
\begin{align*}
C^\infty_c(V^r)^{I_r}[\kappa_r^\epsilon]&\to(\pi_{W_r,\epsilon}^{\overleftarrow\sigma})^{I_r}[\kappa_r^\epsilon]\boxtimes\Theta(\pi,V)
\simeq\Theta(\pi,V), \\
C^\infty_c(V^r)^{I_r}[\kappa_r^\epsilon]&\to(\pi_{W_r,\epsilon}^{\overleftarrow\sigma})^{I_r}[\kappa_r^\epsilon]\boxtimes\theta(\pi,V)
\simeq\theta(\pi,V),
\end{align*}
induced from $\Theta_\pi$ and $\theta_\pi$, respectively, both of which are surjective. Under Assumption \ref{as:invariant}, we know that
$\Theta(\pi,V)$ and $\theta(\pi,V)$ are generated by $\Theta_\pi(\CF_{\Lambda_V^r})$ and $\theta_\pi(\CF_{\Lambda_V^r})$, respectively. In particular, $\theta(\pi,V)^{L_V}\neq\{0\}$. By Proposition \ref{pr:annihilator}, we have $\theta(\pi,V)\simeq\pi_V^{\overrightarrow\sigma}$, and $\Theta(\pi,V)\simeq\rI_V^{\overrightarrow\sigma}$ if $\rI_V^{\overrightarrow\sigma}$ is irreducible. Thus, (2) is proved.

Now it remains to show the claim \eqref{eq:theta1}. By Lemma \ref{le:theta}, Proposition \ref{pr:zeta}, and Theorem \ref{th:gcd}, we have
\[
\frac{Z(\xi_r^{\overleftarrow\sigma},f^{(s)}_{\CF_{\Lambda_V^r}\otimes\CF_{\Lambda_V^r}})}{L(s+\tfrac{1}{2},\pi)}
=\frac{c_\epsilon^r(s)}{b_{2r}(s)},
\]
which is holomorphic and nonvanishing for $\RE s\geq 0$. Thus, \eqref{eq:theta1} holds when $s_0\geq 0$. Now suppose that $s_0<0$. Then $s_0$ belongs to $\{-1,\dots,-r\}$ and $\{-1,\dots,-r+1\}$ when $\epsilon=+$ and $\epsilon=-$, respectively. It is clear that in both cases, $c_\epsilon^r(s)/b_{2r}(s)$ has a simple zero at $s_0$. Thus, \eqref{eq:theta1} holds when $s_0<0$ as well.

The theorem is proved.
\end{proof}

Now we can conclude this note with the proof of Theorem \ref{th:support}.

\begin{proof}[Proof of Theorem \ref{th:support}]
Part (1) has been proved in Proposition \ref{pr:annihilator}. In particular, $\Sph(V^r)$ is a module over $(\cH_{W_r}^\epsilon\otimes\cH_V)/\cI_{W_r,V}$.

For (2), it suffices to consider the $\cH_{W_r}^\epsilon\otimes\cH_V$-module $\Sph(V^r)$. We have $\CF_{\Lambda_V^r}\in\Sph(V^r)$ by Lemma \ref{le:generator}. By Proposition \ref{pr:invariant} and Theorem \ref{th:theta}(2), for every $\sigma\in(\dC/\tfrac{\pi i}{\log q}\dZ)^{\min\{r,s\}}$, the map $\theta_\pi$ \eqref{eq:theta} induces a map
\[
\Sph(V^r)\to(\pi_{W_r,\epsilon}^{\overleftarrow\sigma})^{I_r}[\kappa_r^\epsilon]\boxtimes(\pi_V^{\overrightarrow\sigma})^{L_V}
\]
of $\cH_{W_r}^\epsilon\otimes\cH_V$-modules, under which the image of $\CF_{\Lambda_V^r}$ is nonzero. Therefore, the section $\CF_{\Lambda_V^r}$ of the module $\Sph(V^r)$ is nowhere vanishing on $\Spec(\cH_{W_r}^\epsilon\otimes\cH_V)/\cI_{W_r,V}$, that is, $\CF_{\Lambda_V^r}$ is free. Thus, (2) is proved.
\end{proof}

\if false

\newpage

\textbf{Reply to the referee:}

\begin{enumerate}
  \item Add a footnote.

  \item Corrected.

  \item Add an explanation.

  \item Corrected.

  \item Corrected.

  \item Yes, if $\chi$ is unramified. However, for arithmetic application, we need the global $\chi$ to be orthogonal-conjugate selfdual, which makes $\chi$ trivial at a ``good'' inert place.

  \item Corrected.

  \item This is a very good question. I believe the answer is YES, but I do not see immediately a proof unless $n=2$ in which case one can argue by the genus of the Shimura curve. However, in the global application, namely, the arithmetic inner product formula in my recent work with Chao Li, we have no restriction for the local component at a split place (the current article still requires it to be a principal series, possibly ramified; but we will lift it in a subsequent article). Then it is not hard to see that such automorphic representation exists.

  \item Add a reference.

  \item Add more details in that proof.

  \item Add a reference for how this correspondence is obtained. This remark is not really used in this article, just providing a viewpoint.

  \item Add a reference.

  \item Use a different notation.

  \item Add a footnote.

  \item Corrected.

  \item Your guess is correct. In fact, this has already been proved in \cite{LTXZZ}*{Lemma~C.2.3}.

  \item I think the answer is the (hermitian) Siegel parahoric subgroup. The question boils down to the study of the so-called Tate--Thompson representation $\Omega_{2r}$ in \cite{LTXZZ}*{Section~C.2}, and the answer is implicitly contained in \cite{LTXZZ}*{Proposition~C.2.1}.
\end{enumerate}

\fi

\end{document}